\documentclass[12pt,reqno]{amsart}

\usepackage[left=3.9cm,right=3.9cm]{geometry}
\usepackage{amssymb}
\usepackage{graphicx}
\usepackage{amscd}
\usepackage[pagebackref]{hyperref}
\usepackage{color}
\usepackage{tabularx}
\usepackage[table]{xcolor}
\usepackage{float}
\usepackage{graphics,amsmath,amssymb}
\usepackage{amsthm}
\usepackage{amsfonts}
\usepackage{latexsym}
\usepackage{epsf}
\usepackage{enumerate}
\usepackage{xifthen}
\usepackage{mathrsfs}
\usepackage{dsfont}
\usepackage{makecell}
\usepackage[FIGTOPCAP]{subfigure}
\usepackage{amsmath}
\allowdisplaybreaks[4]
\usepackage{listings}
\usepackage{etoolbox}
\usepackage{fancyhdr}
\usepackage{pdflscape}
\usepackage[title,toc,titletoc]{appendix}
\usepackage{enumitem}
\usepackage[noadjust]{cite}
\usepackage{diagbox}
\usepackage[object=vectorian]{pgfornament} 
\usepackage{lipsum,tikz}
\usepackage{easyReview}

\hypersetup{
	colorlinks=true, 
	linktoc=all, 
	linkcolor=blue} 

\numberwithin{equation}{section}

\theoremstyle{theorem}
\newtheorem{theorem}{Theorem}[section]
\newtheorem*{theorem*}{Theorem}

\newtheorem{lemma}[theorem]{Lemma}

\newtheorem{conjecture}{Conjecture}[section]

\newtheorem{innercustomgeneric}{\customgenericname}
\providecommand{\customgenericname}{}
\newcommand{\newcustomtheorem}[2]{%
	\newenvironment{#1}[1]
	{%
		\renewcommand\customgenericname{#2}%
		\renewcommand\theinnercustomgeneric{##1}%
		\innercustomgeneric
	}
	{\endinnercustomgeneric}
}
\newcustomtheorem{ctheorem}{Theorem}
\newcustomtheorem{clemma}{Lemma}

\theoremstyle{definition}

\newtheorem*{definition*}{Definition}

\newtheorem*{example*}{Example}
\newtheorem*{examples*}{Examples}

\newtheorem*{remark*}{Remark}
\newtheorem*{remarks*}{Remarks}
\newtheorem*{notation*}{Notation}
\newtheorem*{note*}{Note}

\makeatletter
\patchcmd{\section}{\scshape}{\bfseries\boldmath}{}{}
\patchcmd{\subsection}{\bfseries}{\bfseries\boldmath}{}{}
\renewcommand{\@secnumfont}{\bfseries}
\makeatother

\newcommand{\cZ}{\mathcal{Z}}

\newcommand{\cV}{\mathcal{V}}
\newcommand{\tcZ}{\mathcal{Z}}

\newcommand{\NZ}{N\!Z}
\newcommand{\hZ}{\widehat{Z}}
\newcommand{\hNZ}{\widehat{N\!Z}}
\newcommand{\tR}{\widetilde{R}}

\newcommand{\qbinom}[2]{{#1\brack #2}}

\newcommand{\LHS}{\operatorname{LHS}}
\newcommand{\RHS}{\operatorname{RHS}}

\newcommand{\Quot}{\operatorname{Quot}}
\newcommand{\Var}{\operatorname{Var}}
\newcommand{\Coh}{\operatorname{Coh}}
\newcommand{\card}{\operatorname{card}}

\title{Multiple Rogers--Ramanujan type identities for torus links}

\author[S. Chern]{Shane Chern}
\address{Fakult\"at f\"ur Mathematik, Universit\"at Wien, Oskar-Morgenstern-Platz 1, Wien 1090, Austria}
\email{chenxiaohang92@gmail.com, xiaohangc92@univie.ac.at}

\date{}



\begin{document}
	
\sloppy
	
\maketitle

\begin{abstract}
	
	In this paper, we establish simple $k$-fold summation expressions for the Quot and motivic Cohen--Lenstra zeta functions associated with the $(2,2k)$ torus links. Such expressions lead us to some multiple Rogers--Ramanujan type identities and their finitizations, thereby confirming a conjecture of Huang and Jiang. Several other properties of the two zeta functions will be examined as well.
	
\end{abstract}

\section{Introduction}

The main objective of this paper revolves around some conjectural Rogers--Ramanujan type identities arising from algebraic geometry. To embark on our journey, we let $\mathbb{K}$ be a fixed field. Now given a certain $\mathbb{K}$-curve at a $\mathbb{K}$-point, we let $R$ be the complete local ring of its germ and $\tR$ the normalization of $R$, and assume that $E$ is a finitely generated $R$-module; this setting localizes reduced varieties $X$ over $\mathbb{K}$ and coherent sheaves $\mathcal{E}$ on $X$. We further denote by $\Quot_{E,n}$ the Quot scheme parametrizing $R$-submodules of $E$ of $\mathbb{K}$-codimension $n$. What lies at the heart of our work is the \emph{Quot zeta function}:
\begin{align*}
	Z_E^R(t) = Z_E(t) := \sum_{n\ge 0} [\Quot_{E,n}] t^n,
\end{align*}
where the motive $[V]$ denotes the class of $V$ in the Grothendieck ring $K_0(\Var_\mathbb{K})$ of $\mathbb{K}$-varieties for $V$ a $\mathbb{K}$-scheme.

Investigations on $Z_R^R(t)$ and $Z_{\tR}^R(t)$ have been widely performed in the past, and among those the beautiful Hilb-vs-Quot conjecture \cite{KT2023} predicts the connection between $Z_R^R(t)$ and $Z_{\tR}^R(t)$. What is then highlighted in a recent work of Huang and Jiang \cite{HJ2023} is a high-rank generalization in the sense that $E$ is taken to be a \emph{torsion-free module} of rank $N$ over $R$, meaning that $E$ is injective to $E \oplus_R \operatorname{Frac}(R) \simeq \operatorname{Frac}(R)^N$ with $\operatorname{Frac}(R)$ the total fraction ring of $R$.

Notably, the \emph{rationality theorem} of Huang and Jiang \cite[Theorem 1.3]{HJ2023} asserts that under the assumption that $\tR \simeq \mathbb{K}[[T]]^s$ with $s$ the branching number of $R$, we have that $Z_E^R(t)/Z_{\tR^{\oplus N}}^{\tR}(t)$ is a polynomial in $t$ for any torsion-free module $E$ of rank $N$ over $R$. Here, it is known \cite{Bif1989} that
\begin{align}\label{eq:Z-deno}
	Z_{\tR^{\oplus N}}^{\tR}(t) = \prod_{j= 0}^{N-1} \frac{1}{(1-t\mathbb{L}^j)^s},
\end{align}
where $\mathbb{L}:=[\mathbb{A}^1]$ is the \emph{Lefschetz motive}. This rationality theorem leads one to focus on the \emph{numerator part} of $Z_E^R(t)$:
\begin{align}\label{eq:NZ-def}
	\NZ_E^R(t) = \NZ_E(t) := \frac{Z_E^R(t)}{Z_{\tR^{\oplus N}}^{\tR}(t)}.
\end{align}

In addition, a generalization of the important Cohen--Lenstra zeta function \cite{CL1984} was recently introduced by Huang \cite{Hua2023} to the motivic version. Briefly speaking, by denoting $\Coh_n(R)$ the stack of $R$-modules of $\mathbb{K}$-dimension $n$, the \emph{motivic Cohen--Lenstra zeta function} is defined by
\begin{align*}
	\hZ_R(t) := \sum_{n\ge 0} [\Coh_n(R)] t^n.
\end{align*}
A remarkable result in \cite[Theorem 1.12]{HJ2023} connects the motivic Cohen--Lenstra zeta functions and the limiting case of the Quot zeta functions. To be specific, if $R$ is a complete local $\mathbb{K}$-algebra of finite type with residue field $\mathbb{K}$, then
\begin{align}\label{eq:zeta-limit}
	\hZ_R(t) = \lim_{N\to \infty} Z_{R^{\oplus N}}(t\mathbb{L}^{-N}).
\end{align}
Analogous to \eqref{eq:NZ-def}, we may also define the \emph{numerator part}:
\begin{align}
	\hNZ_R(t) := \frac{\hZ_R(t)}{\hZ_{\tR}(t)},
\end{align}
while we note from \eqref{eq:Z-deno} that
\begin{align}
	\hZ_{\tR}(t) = \prod_{j\ge 0} \frac{1}{(1-t\mathbb{L}^{-j-1})^s}.
\end{align}

The above objects have profound applications to matrix Diophantine equations when the field $\mathbb{K}$ is finite, namely, $\mathbb{K}\simeq \mathbb{F}_q$ for $q$ a prime power. As shown in \cite[p.~40, Proposition 4.3]{Hua2023}, for $R = \mathbb{F}_q[x,y]/f(x,y)$ where $f$ is a polynomial,
\begin{align*}
	\hZ_R(t) = \sum_{n\ge 0} \frac{\card \mathscr{M}_n}{\card \operatorname{GL}_n(\mathbb{F}_q)} t^n,
\end{align*}
where $\mathscr{M}_n$ is the following set of matrix pairs over $\mathbb{F}_q$:
\begin{align*}
	\mathscr{M}_{n} := \big\{(A,B)\in \operatorname{Mat}_n(\mathbb{F}_q)^2: \text{$AB=BA$ and $f(A,B)=0$}\big\}.
\end{align*}
In view of this generating series for the enumeration of \emph{commuting} matrix pairs $(A,B)$ over $\mathbb{F}_q$ satisfying the additional restriction that $f(A,B)=0$, it becomes extremely meaningful to chase nice expressions for motivic Cohen--Lenstra zeta functions.

In \cite{HJ2023}, planar singularities associated with the $(2,n)$ torus knots and links are particularly considered. That is to say, we define $R^{(2,2k+1)}$ to be the germ of the variety $y^2 = x^{2k+1}$ and $R^{(2,2k)}$ the germ of $y(y-x^k)=0$; for the latter case, if $\mathbb{K}$ is not of characteristic two, then $R^{(2,2k)}$ also admits the variety $y^2 = x^{2k}$. It is known \cite[\S{}8.2]{HJ2023} that the branching number of $R^{(2,2k+1)}$ is $1$, while for $R^{(2,2k)}$, it is $2$.

One important result in \cite{HJ2023} is the following formula for the motivic Cohen--Lenstra zeta function $\hNZ_{R^{(2,2k+1)}}(t)$ \cite[Theorem 1.13]{HJ2023}:
\begin{align}\label{eq:2k+1-sum}
	\hNZ_{R^{(2,2k+1)}}(t) = \sum_{n_1,\ldots,n_k\ge 0} \frac{t^{\sum_{i=1}^k 2n_i} \mathbb{L}^{-\sum_{i=1}^k n_i^2}}{\prod_{i=1}^k \prod_{j=1}^{n_i-n_{i-1}}(1-\mathbb{L}^{-j})},
\end{align}
where we put $n_0 := 0$. We may further specialize at $t = \pm 1$ \cite[eq.~(1.22)]{HJ2023}:
\begin{align}\label{eq:2k+1-prod}
	&\hNZ_{R^{(2,2k+1)}}(\pm 1)\notag\\
	&\qquad = \prod_{j\ge 0} \frac{(1-\mathbb{L}^{-(2k+3)j-(k+1)})(1-\mathbb{L}^{-(2k+3)j-(k+2)})(1-\mathbb{L}^{-(2k+3)j-(2k+3)})}{1-\mathbb{L}^{-j-1}}.
\end{align}

On the other hand, the expression for $\hNZ_{R^{(2,2k)}}(t)$ shown in \cite{HJ2023} is unfortunately not satisfactory, as will be seen in \eqref{eq:2k-sum-double}. However, when $t=1$, the following beautiful equality is given in \cite[Theorem 1.16]{HJ2023}.

\begin{theorem}[Huang--Jiang]
	For any positive integer $k$,
	\begin{align}\label{eq:2k-at-1}
		\hNZ_{R^{(2,2k)}}(1) = 1.
	\end{align}
\end{theorem}

Meanwhile, Huang and Jiang \cite[Conjecture 1.17]{HJ2023} also proposed a neat conjectural evaluation at $t=-1$.

\begin{conjecture}[Huang--Jiang]\label{conj:HJ-(-1)}
	For any positive integer $k$,
	\begin{align}\label{eq:2k-at-(-1)}
		\hNZ_{R^{(2,2k)}}(-1) = \prod_{j\ge 1} \frac{(1-\mathbb{L}^{-2j})(1-\mathbb{L}^{-(k+1)j})^2}{(1-\mathbb{L}^{-j})^2(1-\mathbb{L}^{-(2k+2)j})}.
	\end{align}
\end{conjecture}

A glimpse at the sum in \eqref{eq:2k+1-sum} and the product in \eqref{eq:2k+1-prod} readily reminds one of identities of Rogers--Ramanujan type. Before moving on to this topic, we adopt the conventional \emph{$q$-Pochhammer symbols} for $n\in\mathbb{N}\cup\{\infty\}$:
\begin{align*}
	(A;q)_n&:=\prod_{j=0}^{n-1} (1-A q^j),\\
	(A_1,A_2,\ldots,A_r;q)_n&:=(A_1;q)_n(A_2;q)_n\cdots (A_r;q)_n,
\end{align*}
and the \emph{$q$-binomial coefficients}:
\begin{align*}
	\qbinom{N}{M}_q:=\begin{cases}
		\dfrac{(q;q)_N}{(q;q)_M(q;q)_{N-M}}, & \text{if $0\le M\le N$},\\[10pt]
		0, & \text{otherwise}.
	\end{cases}
\end{align*}

The famous \emph{Rogers--Ramanujan identities} refer to the following two $q$-series equalities:
\begin{align*}
	\sum_{n\ge 0} \frac{q^{n^2}}{(q;q)_n} &= \frac{1}{(q,q^4;q^5)_\infty},\\
	\sum_{n\ge 0} \frac{q^{n^2+n}}{(q;q)_n} &= \frac{1}{(q^2,q^3;q^5)_\infty}.
\end{align*}
They were first established by Rogers \cite{Rog1894} in 1894, and had been unfortunately overlooked until Ramanujan's rediscovery \cite{Ram1914} two decades later. Around that time, a new proof was provided jointly by Ramanujan and Rogers \cite{RR1919} and another two fundamentally different proofs were offered by Schur \cite{Sch1917}. Since then, we usually refer to $q$-series relations of the form ``$\text{sum side} = \text{product side}$'' as \emph{Roger--Ramanujan type identities}. Now the equality between \eqref{eq:2k+1-sum} and \eqref{eq:2k+1-prod} is
\begin{align*}
	\sum_{n_1,\ldots,n_{k}\ge 0}\frac{q^{\sum_{i=1}^k n_i^2}}{(q;q)_{n_k-n_{k-1}}\cdots (q;q)_{n_{2}-n_1} (q;q)_{n_1}} = \frac{(q^{k+1},q^{k+2},q^{2k+3};q^{2k+3})_\infty}{(q;q)_\infty}.
\end{align*}
This identity was first discovered by Andrews \cite{And1974} in a more general form, which also serves as an analytic counterpart of a partition-theoretic relation due to Gordon \cite{Gor1961}.

To connect the motivic Cohen--Lenstra zeta functions for the $(2,2k)$ torus links with identities of Rogers--Ramanujan type, it is necessary to find a sum-like expression for $\hNZ_{R^{(2,2k)}}(t)$. Fortunately, this can be achieved by means of the \emph{Hall--Littlewood polynomials}:
\begin{align*}
	g_{\boldsymbol{s}}^{\boldsymbol{r}}(q) := q^{\sum_{i=1}^k s_i(r_i-s_i)} \prod_{i=1}^k \qbinom{r_i-s_{i-1}}{r_i-s_i}_{q^{-1}},
\end{align*}
where $\boldsymbol{r}=(r_1,\ldots,r_k)$ and $\boldsymbol{s}=(s_1,\ldots,s_k)$ are weakly \emph{increasing}\footnote{In \cite{HJ2023}, the sequences $\boldsymbol{r}$ and $\boldsymbol{s}$ are weakly \emph{decreasing} so that the top entries of the $q$-binomial coefficients are $r_i-s_{i+1}$, but for our convenience in the current work, we flip them over.} sequences of nonnegative integers, while we assume that $s_0:=0$. Then \cite[Theorem 1.14]{HJ2023} asserts that
\begin{align}\label{eq:2k-sum-double}
	\hNZ_{R^{(2,2k)}}(t)&= (t\mathbb{L}^{-1};\mathbb{L}^{-1})_\infty^2\sum_{\boldsymbol{r},\boldsymbol{s}} \frac{t^{\sum_{i=1}^k (2r_i-s_i)} \mathbb{L}^{-\sum_{i=1}^k r_i^2} g_{\boldsymbol{s}}^{\boldsymbol{r}}(\mathbb{L})}{(t\mathbb{L}^{-1};\mathbb{L}^{-1})_{r_1}^2(\mathbb{L}^{-1};\mathbb{L}^{-1})_{s_1}}\notag\\
	&\quad\times \frac{1}{(\mathbb{L}^{-1};\mathbb{L}^{-1})_{r_{k}-r_{k-1}}\cdots (\mathbb{L}^{-1};\mathbb{L}^{-1})_{r_2-r_1}}.
\end{align}
To facilitate our analysis, we define
\begin{align}
	\cZ_k(t,q)&:=\sum_{\substack{r_k\ge \cdots\ge r_1\ge 0\\s_k\ge \cdots\ge s_1\ge 0}} \frac{t^{\sum_{i=1}^k (2r_i-s_i)} q^{\sum_{i=1}^k (r_i^2-r_is_i+s_i^2)}}{(q;q)_{r_{k}-r_{k-1}}\cdots (q;q)_{r_2-r_1}(tq;q)_{r_1}^2(q;q)_{s_1}}\notag\\
	&\quad\,\times \qbinom{r_k-s_{k-1}}{r_k-s_k}_q\qbinom{r_{k-1}-s_{k-2}}{r_{k-1}-s_{k-1}}_q\cdots \qbinom{r_{2}-s_1}{r_{2}-s_{2}}_q\qbinom{r_1}{r_1-s_1}_q.
\end{align}
It is then clear that
\begin{align}
	\hNZ_{R^{(2,2k)}}(t)|_{\mathbb{L}\mapsto q^{-1}} = (tq;q)_\infty^2 \cZ_k(t,q).
\end{align}

Note that the proof of \eqref{eq:2k-at-1} in \cite{HJ2023} relies heavily on hardcore techniques in algebraic geometry. Recently, in a private communication with Yifeng Huang, one of the authors of \cite{HJ2023}, a purely $q$-theoretic proof of \eqref{eq:2k-at-1} was requested. This is the starting point of our work.

\begin{theorem}\label{th:S-m-infty}
	For any positive integer $k$,
	\begin{align}\label{eq:S-m-infty}
		\cZ_k(1,q) = \frac{1}{(q;q)_\infty^2}.
	\end{align}
	Consequently, \eqref{eq:2k-at-1} is true.
\end{theorem}

We will show that this relation is indeed a consequence of the following multiple Rogers--Ramanujan type identity.

\begin{theorem}\label{th:multi-sum-infty}
	For any nonnegative integers $d_1,\ldots,d_k$,
	\begin{align}\label{eq:multi-sum-infty}
		\frac{1}{(aq;q)_\infty} &= \sum_{n_1,\ldots,n_k\ge 0} \frac{a^{\sum_{i=1}^k n_i}q^{\sum_{i=1}^k n_i^2+\sum_{i=1}^k(d_1+\cdots+d_i)n_i}}{(q;q)_{n_k-n_{k-1}+d_k} \cdots (q;q)_{n_2-n_1+d_2}}\notag\\
		&\quad\times \frac{(q;q)_{n_k+d_k}\cdots (q;q)_{n_2+d_2}}{ (q;q)_{n_k} \cdots (q;q)_{n_1} (aq;q)_{n_1+d_1}}.
	\end{align}
\end{theorem}

Notably, letting $d_1=\cdots=d_k=0$, the above becomes
\begin{align}\label{eq:multi-sum-infty-0}
	\frac{1}{(aq;q)_\infty} &= \sum_{n_1,\ldots,n_k\ge 0} \frac{a^{\sum_{i=1}^k n_i}q^{\sum_{i=1}^k n_i^2}}{(q;q)_{n_k-n_{k-1}} \cdots (q;q)_{n_2-n_1} (q;q)_{n_1}(aq;q)_{n_1}}.
\end{align}
This is an instance of \cite[p.~30, eq.~(3.44)]{And1986} by choosing the following \emph{Bailey pair} relative to $(a,q)$ \cite[p.~25, eq.~(3.27)]{And1986}:
\begin{align*}
	\alpha_n = \begin{cases}
		1, & n=0,\\
		0, & n\ge 1,
	\end{cases}\qquad\text{and}\qquad \beta_n = \frac{1}{(q;q)_n(aq;q)_n}.
\end{align*}

One may wonder if the same method works for the evaluation of $\hNZ_{R^{(2,2k)}}(t)$ at $t=-1$ so as to attack Conjecture \ref{conj:HJ-(-1)}. Sadly, this is not the case. Now a natural idea is to figure out a simpler expression for $\hNZ_{R^{(2,2k)}}(t)$, by reducing the number of summation folds from $2k$ to $k$, thereby yielding an analog to the case of $(2,2k+1)$ torus knots in \eqref{eq:2k+1-sum}.

\begin{theorem}\label{th:Z-inf-expression}
	For any positive integer $k$,
	\begin{align}\label{eq:Z-inf-expression}
		\cZ_k(t,q) &= \frac{1}{(tq;q)_\infty} \sum_{n_1,\ldots,n_k\ge 0} \frac{t^{\sum_{i=1}^k 2n_i} q^{\sum_{i=1}^k n_i^2}}{(q;q)_{n_k-n_{k-1}}\cdots (q;q)_{n_2-n_1}(q;q)_{n_1}(tq;q)_{n_1}}.
	\end{align}
	Consequently,
	\begin{align}\label{eq:hNZ-2k}
		\hNZ_{R^{(2,2k)}}(t)&= (t\mathbb{L}^{-1};\mathbb{L}^{-1})_\infty \sum_{n_1,\ldots,n_k\ge 0} \frac{t^{\sum_{i=1}^k 2n_i} \mathbb{L}^{-\sum_{i=1}^k n_i^2}}{(t\mathbb{L}^{-1};\mathbb{L}^{-1})_{n_1}}\notag\\
		&\quad\times \frac{1}{(\mathbb{L}^{-1};\mathbb{L}^{-1})_{n_k-n_{k-1}}\cdots (\mathbb{L}^{-1};\mathbb{L}^{-1})_{n_2-n_1}(\mathbb{L}^{-1};\mathbb{L}^{-1})_{n_1}}.
	\end{align}
\end{theorem}

Now the evaluation at $t=-1$ becomes immediate.

\begin{theorem}\label{th:S-m-(-1)-infty}
	For any positive integer $k$,
	\begin{align}\label{eq:S-m-(-1)-infty}
		\cZ_k(-1,q) = \frac{(q^{k+1};q^{k+1})_\infty^2}{(q^2;q^2)_\infty (q^{2k+2};q^{2k+2})_\infty}.
	\end{align}
	Consequently, \eqref{eq:2k-at-(-1)} in Huang--Jiang's Conjecture \ref{conj:HJ-(-1)} is true.
\end{theorem}

It is remarkable that as a middle step in our proof of Theorems \ref{th:Z-inf-expression} and \ref{th:S-m-(-1)-infty}, we observe that \eqref{eq:S-m-(-1)-infty} is closely tied with a more surprising identity.

\begin{theorem}\label{th:strange-A2-sum}
	For any positive integer $k$,
	\begin{align}\label{eq:strange-A2-sum}
		&\frac{(q^2;q^2)_\infty (q^{k+1};q^{k+1})_\infty^2}{(q;q)_\infty^3 (q^{2k+2};q^{2k+2})_\infty}\notag\\
		&\qquad\qquad = \sum_{\substack{m_1,\ldots,m_k\ge 0\\n_1,\ldots,n_k\ge 0}} \frac{(-1)^{\sum_{i=1}^k m_i} q^{-n_1^2+n_1+\sum_{i=1}^k (m_i^2+m_in_i+n_i^2)} (-1;q)_{n_1}^2}{ (q;q)_{m_k}(q;q)_{m_1} (q;q)_{n_1}}\notag\\
		&\qquad\qquad\quad\times \qbinom{m_k}{m_{k-1}}_q \qbinom{m_{k-1}}{m_{k-2}}_q \cdots \qbinom{m_2}{m_1}_q \qbinom{n_1}{n_2}_q \cdots \qbinom{n_{k-2}}{n_{k-1}}_q \qbinom{n_{k-1}}{n_k}_q.
	\end{align}
\end{theorem}

It is an easy observation that the format of \eqref{eq:strange-A2-sum} resembles $\mathrm{A}_2$ Rogers--Ramanujan type identities introduced by Andrews, Schilling and Warnaar such as \cite[p.~694, eq.~(5.22)]{ASW1999}:
\begin{align*}
	&\frac{(q^{k+1},q^{k+1},q^{k+2},q^{2k+2},q^{2k+3},q^{2k+3},q^{3k+4},q^{3k+4};q^{3k+4})_\infty}{(q;q)_\infty^3}\notag\\
	&\qquad\qquad\qquad = \sum_{\substack{m_1,\ldots,m_k\ge 0\\n_1,\ldots,n_k\ge 0}} \frac{q^{\sum_{i=1}^k (m_i^2-m_in_i+n_i^2)} (1-q^{m_1+n_1+1})}{ (q;q)_{m_k}(q;q)_{n_k}(q;q)_{m_1+n_1+1}}\notag\\
	&\qquad\qquad\qquad\quad\times \qbinom{m_k}{m_{k-1}}_q \qbinom{m_{k-1}}{m_{k-2}}_q \cdots \qbinom{m_2}{m_1}_q \qbinom{n_k}{n_{k-1}}_q \qbinom{n_{k-1}}{n_{k-2}}_q \cdots \qbinom{n_2}{n_1}_q.
\end{align*}
However, the fact that the summation indices $m_i$ and $n_i$ in \eqref{eq:strange-A2-sum} are in reverse order makes it fundamentally different from the above $\mathrm{A}_2$ Rogers--Ramanujan type identity. This is also to some extent a hint for \eqref{eq:strange-A2-sum} lacking the usual ``triple the number of folds'' phenomenon (from the $\mathrm{A}_2$ Macdonald identity \cite{Mac1972}) for the modulus.

Next, let us recall from \eqref{eq:zeta-limit} that the motivic Cohen--Lenstra zeta functions are the limiting case of Quot zeta functions. It turns out that for the $(2,2k)$ torus links, the expression of the Quot zeta function $Z_{R^{(2,2k)\oplus N}}(t)$, or equivalently its numerator part $\NZ_{R^{(2,2k)\oplus N}}(t)$, is also known, as given in \cite[Theorem 1.10]{HJ2023}:
\begin{align*}
	&\NZ_{R^{(2,2k)\oplus N}}(t)\\
	&\ = t^{2N}\mathbb{L}^{N^2-N}(\mathbb{L}^{-1};\mathbb{L}^{-1})_N(t^{-1};\mathbb{L}^{-1})_N^2\\
	&\ \quad\times\sum_{\substack{r_k\ge \cdots\ge r_1\ge 0\\s_k\ge \cdots\ge s_1\ge 0}} \frac{(t\mathbb{L}^N)^{\sum_{i=1}^k (2r_i-s_i)} \mathbb{L}^{-\sum_{i=1}^k (r_i^2-r_is_i+s_i^2)}}{(\mathbb{L}^{-1};\mathbb{L}^{-1})_{N-r_k}(\mathbb{L}^{-1};\mathbb{L}^{-1})_{r_{k}-r_{k-1}}\cdots (\mathbb{L}^{-1};\mathbb{L}^{-1})_{r_2-r_1}}\notag\\
	&\ \quad\times \frac{1}{(t\mathbb{L}^{N-1};\mathbb{L}^{-1})_{r_1}^2(\mathbb{L}^{-1};\mathbb{L}^{-1})_{s_1}} \qbinom{r_k-s_{k-1}}{r_k-s_k}_{\mathbb{L}^{-1}}\cdots \qbinom{r_{2}-s_1}{r_{2}-s_{2}}_{\mathbb{L}^{-1}}\qbinom{r_1}{r_1-s_1}_{\mathbb{L}^{-1}},
\end{align*}
where we have used the relation:
\begin{align*}
	(t;\mathbb{L})_{N-r_1} = (-t)^N \mathbb{L}^{\binom{N}{2}} \frac{(t^{-1};\mathbb{L}^{-1})_N^2}{(t\mathbb{L}^{N-1};\mathbb{L}^{-1})_{r_1}}.
\end{align*}

Hence, we are strongly suggested to consider the following truncation of $\cZ_k(t,q)$:
\begin{align}
	\tcZ_k(N;t,q)&:= \sum_{\substack{r_k\ge \cdots\ge r_1\ge 0\\s_k\ge \cdots\ge s_1\ge 0}} \frac{t^{\sum_{i=1}^k (2r_i-s_i)} q^{\sum_{i=1}^k (r_i^2-r_is_i+s_i^2)}}{(q;q)_{N-r_k}(q;q)_{r_2-r_1}\cdots (q;q)_{r_{k}-r_{k-1}}(tq;q)_{r_1}^2(q;q)_{s_1}}\notag\\
	&\,\quad\times \qbinom{r_k-s_{k-1}}{r_k-s_k}_q\qbinom{r_{k-1}-s_{k-2}}{r_{k-1}-s_{k-1}}_q\cdots \qbinom{r_{2}-s_1}{r_{2}-s_{2}}_q\qbinom{r_1}{r_1-s_1}_q.
\end{align}
It is notable that this sum is \emph{finite} as the factor $1/(q;q)_{N-r_k}$ requires $r_k\le N$ to ensure its nonvanishing. Also, at the limit $N\to \infty$,
\begin{align}\label{eq:Z-finite-limit}
	\lim_{N\to \infty} \tcZ_k(N;t,q) = \frac{\cZ_k(t,q)}{(q;q)_\infty}.
\end{align}
Furthermore,
\begin{align}\label{eq:NZ-cZ}
	\NZ_{R^{(2,2k)\oplus N}}(t) = t^{2N}\mathbb{L}^{N^2-N}(\mathbb{L}^{-1};\mathbb{L}^{-1})_N(t^{-1};\mathbb{L}^{-1})_N^2 \tcZ_k(N;t\mathbb{L}^{N},\mathbb{L}^{-1}).
\end{align}

Now our objective is to show that Theorem \ref{th:Z-inf-expression} can be finitized as follows.

\begin{theorem}\label{th:Z-expression}
	For any nonnegative integer $N$,
	\begin{align}\label{eq:Z-expression}
		\tcZ_k(N;t,q) &= \frac{1}{(tq;q)_N} \sum_{n_1,\ldots,n_k\ge 0} \frac{t^{\sum_{i=1}^k 2n_i} q^{\sum_{i=1}^k n_i^2}}{(q;q)_{N-n_k}(q;q)_{n_k}(tq;q)_{n_1}}\notag\\
		&\quad\times  \qbinom{n_k}{n_{k-1}}_q\qbinom{n_{k-1}}{n_{k-2}}_q\cdots \qbinom{n_2}{n_1}_q.
	\end{align}
	Consequently,
	\begin{align}\label{eq:NZ-2k}
		\NZ_{R^{(2,2k)\oplus N}}(t) &= (t\mathbb{L}^{N-1};\mathbb{L}^{-1})_N \sum_{n_1,\ldots,n_k\ge 0} \frac{(t\mathbb{L}^N)^{\sum_{i=1}^k 2n_i} \mathbb{L}^{-\sum_{i=1}^k n_i^2}}{(t\mathbb{L}^{N-1};\mathbb{L}^{-1})_{n_1}}\notag\\
		&\quad\times  \qbinom{N}{n_{k}}_{\mathbb{L}^{-1}}\qbinom{n_k}{n_{k-1}}_{\mathbb{L}^{-1}}\qbinom{n_{k-1}}{n_{k-2}}_{\mathbb{L}^{-1}}\cdots \qbinom{n_2}{n_1}_{\mathbb{L}^{-1}}.
	\end{align}
\end{theorem}

Here \eqref{eq:NZ-2k} follows from \eqref{eq:NZ-cZ} with an application of the relation:
\begin{align*}
	(t\mathbb{L}^{N-1};\mathbb{L}^{-1})_N = (-1)^N t^N \mathbb{L}^{\binom{N}{2}} (t^{-1};\mathbb{L}^{-1})_N.
\end{align*}

Finally, to close this section, we present two implications of our previous results on $\hNZ_{R^{(2,2k)}}(t)$ and $\NZ_{R^{(2,2k)\oplus N}}(t)$.

The first one concerns a remarkable \emph{reflection formula} of Huang and Jiang \cite[Conjecture 1.6 and Theorem 1.7]{HJ2023}. To state this formula, we let $E=\Omega^{\oplus N}$ with $\Omega$ the dualizing module of $R$ under the assumption that $\tR \simeq \mathbb{K}[[T]]^s$. Then \cite[Conjecture 1.6]{HJ2023} predicates that
\begin{align*}
	\NZ_E(t) \overset{?}{=} (t^{2N}\mathbb{L}^{N^2})^\delta \NZ_E(t^{-1}\mathbb{L}^{-N}),
\end{align*}
where $\delta:=\operatorname{dim}_{\mathbb{K}} \tR/R$ is the Serre invariant. This formula remains conjectural but Huang and Jiang proved its point-counting version in \cite[Theorem 1.7]{HJ2023} with recourse to deep techniques in harmonic analysis. Specializing to the case of torus links $R^{(2,2k)}$ and recalling \eqref{eq:NZ-cZ}, it is clear that the reflection
\begin{align}
	\NZ_{R^{(2,2k)\oplus N}}(t) = (t^{2N}\mathbb{L}^{N^2})^k \NZ_{R^{(2,2k)\oplus N}}(t^{-1}\mathbb{L}^{-N})
\end{align}
is equivalent to the following relation, for which we shall offer a purely $q$-theoretic proof.

\begin{theorem}\label{th:HJ-transform}
	For any nonnegative integer $N$,
	\begin{align}\label{eq:HJ-transform}
		\tcZ_k(N;t,q) = \frac{(1-t)^2 q^N (t^{2N} q^{N^2})^{k-1}}{(1-t q^N)^2} \tcZ_k(N;t^{-1}q^{-N},q).
	\end{align}
\end{theorem}

Our second interest revolves around a \emph{nonnegativity conjecture} in \cite[Conjecture 9.13]{HJ2023}.

\begin{conjecture}[Huang--Jiang, Nonnegativity Conjecture]\label{conj-HJ-nonnegativity}
	The zeta functions $\NZ_{R^{(2,2k)\oplus N}}(-t)$ and $\hNZ_{R^{(2,2k)}}(-t)$, as series in $t$ and $\mathbb{L}$, have nonnegative coefficients.
\end{conjecture}

We shall answer it in the affirmative.

\begin{theorem}
	Conjecture \ref{conj-HJ-nonnegativity} is true.
\end{theorem}

\begin{proof}
	We only need to recall \eqref{eq:hNZ-2k} and \eqref{eq:NZ-2k}, and notice the trivial fact that for any nonnegative integer $n$, $(-t\mathbb{L}^{-1};\mathbb{L}^{-1})_N/(-t\mathbb{L}^{-1};\mathbb{L}^{-1})_n$ is a nonnegative bivariate series whenever $N\ge n$ or $N\to \infty$.
\end{proof}

\section{$q$-Series prerequisites}

In this section, we collect some preliminary results on $q$-series. First, we recall \emph{Jacobi's triple product identity} \cite[p.~21, eq.~(2.2.10)]{And1998}:
\begin{lemma}[Jacobi's triple product]
	\begin{align}\label{eq:JTP}
		\sum_{n=-\infty}^\infty z^n q^{n^2} = (-zq,-q/z,q^2;q^2)_\infty.
	\end{align}
\end{lemma}

Next, let the \emph{$q$-hypergeometric function ${}_{r}\phi_s$} be defined by
\begin{align*}
	{}_{r}\phi_s\left(\begin{matrix} A_1,A_2,\ldots,A_r\\ B_1,B_2,\ldots,B_s \end{matrix}; q, z\right):=\sum_{n\ge 0} \frac{(A_1,A_2,\ldots,A_r;q)_n \big((-1)^n q^{\binom{n}{2}}\big)^{s-r+1} z^n}{(q,B_1,B_2,\ldots,B_{s};q)_n}.
\end{align*}

The \emph{$q$-binomial theorem} \cite[p.~354, eq.~(II.3)]{GR2004} is as follows:
\begin{lemma}[$q$-Binomial theorem]
	\begin{align}\label{eq:q-binomial}
		{}_{1} \phi_{0} \left(\begin{matrix}
			a\\
			-
		\end{matrix};q,z\right) = \frac{(az;q)_\infty}{(z;q)_\infty}.
	\end{align}
\end{lemma}

We also require the \emph{$q$-Gau\ss{} sum} \cite[p.~354, eq.~(II.8)]{GR2004}:
\begin{lemma}[$q$-Gau\ss{} sum]
	\begin{align}\label{eq:qGauss}
		{}_{2} \phi_{1} \left(\begin{matrix}
			a,b\\
			c
		\end{matrix};q,\frac{c}{ab}\right) = \frac{(c/a,c/b;q)_\infty}{(c,c/(ab);q)_\infty}.
	\end{align}
\end{lemma}

The \emph{first $q$-Chu--Vandermonde sum} \cite[p.~354, eq.~(II.7)]{GR2004} is a specialization:

\begin{lemma}[First $q$-Chu--Vandermonde sum]
	For any nonnegative integer $N$,
	\begin{align}\label{eq:qCV-1}
		{}_{2} \phi_{1} \left(\begin{matrix}
			a,q^{-N}\\
			c
		\end{matrix};q,\frac{cq^N}{a}\right) = \frac{(c/a;q)_N}{(c;q)_N}.
	\end{align}
\end{lemma}

We then recall \emph{Heine's three transformations} \cite[p.~359, eqs.~(III.1--3)]{GR2004} for ${}_{2} \phi_{1}$ series:

\begin{lemma}[Heine's transformations]
	\begin{align}
		{}_{2} \phi_{1} \left(\begin{matrix}
			a,b\\
			c
		\end{matrix};q,z\right) &= \frac{(b,az;q)_\infty}{(c,z;q)_\infty} {}_{2} \phi_{1} \left(\begin{matrix}
			c/b,z\\
			az
		\end{matrix};q,b\right),\label{eq:Heine1}\\
		{}_{2} \phi_{1} \left(\begin{matrix}
			a,b\\
			c
		\end{matrix};q,z\right) &= \frac{(c/b,bz;q)_\infty}{(c,z;q)_\infty} {}_{2} \phi_{1} \left(\begin{matrix}
			abz/c,b\\
			bz
		\end{matrix};q,\frac{c}{b}\right),\label{eq:Heine2}\\
		{}_{2} \phi_{1} \left(\begin{matrix}
			a,b\\
			c
		\end{matrix};q,z\right) &= \frac{(abz/c;q)_\infty}{(z;q)_\infty} {}_{2} \phi_{1} \left(\begin{matrix}
			c/a,c/b\\
			c
		\end{matrix};q,\frac{abz}{c}\right).\label{eq:Heine3}
	\end{align}
\end{lemma}

Finally, the following transform for ${}_{3} \phi_{2}$ series \cite[p.~359, eq.~(III.9)]{GR2004} is necessary:

\begin{lemma}
	\begin{align}\label{eq:3phi2}
		{}_{3} \phi_{2} \left(\begin{matrix}
			a,b,c\\
			d,e
		\end{matrix};q,\frac{de}{abc}\right) = \frac{(e/a,de/(bc);q)_\infty}{(e,de/(abc);q)_\infty} {}_{3} \phi_{2} \left(\begin{matrix}
			a,d/b,d/c\\
			d,de/(bc)
		\end{matrix};q,\frac{e}{a}\right).
	\end{align}
\end{lemma}

\section{Iteration seed toward Theorem \ref{th:multi-sum-infty}}

Our objective here is to prove Theorem \ref{th:multi-sum-infty} by offering its finitization. To begin with, we establish a simple $q$-hypergeometric transform.

\begin{lemma}
	For any nonnegative integers $M$ and $N$,
	\begin{align}\label{eq:trans-1}
		\sum_{n\ge 0} \frac{a^n q^{n^2 + Mn}}{(q;q)_{N-n} (q;q)_n (aq;q)_{M+n}} = \frac{1}{(q;q)_N (aq;q)_{M+N}}.
	\end{align}
\end{lemma}

\begin{proof}
	We have
	\begin{align*}
		\LHS\eqref{eq:trans-1} &= \frac{1}{(aq;q)_M} \sum_{n=0}^{N} \frac{a^n q^{n^2 + Mn}}{(q;q)_{N-n} (q;q)_n (aq^{M+1};q)_n}\\
		& = \frac{1}{(q;q)_{N} (aq;q)_M} \sum_{n=0}^{N} \frac{a^n q^{n^2 + Mn}\cdot (-1)^n q^{-\binom{n}{2}+Nn} (q^{-N};q)_n}{(q;q)_n (aq^{M+1};q)_n}\\
		& = \frac{1}{(q;q)_{N} (aq;q)_M} \lim_{\tau\to 0} {}_{2}\phi_{1} \left(\begin{matrix}
			1/\tau,q^{-N}\\
			aq^{M+1}
		\end{matrix};q,aq^{M+N+1}\tau\right)\\
		\text{\tiny (by \eqref{eq:qCV-1})}& = \frac{1}{(q;q)_{N} (aq;q)_M} \lim_{\tau\to 0} \frac{(aq^{M+1}\tau;q)_{N}}{(aq^{M+1};q)_{N}}\\
		& = \frac{1}{(q;q)_{N} (aq;q)_{M+N}},
	\end{align*}
	as claimed.
\end{proof}

Now we show that \eqref{eq:trans-1} serves as an iteration seed. Let us start by reformulating it as
\begin{align*}
	\frac{1}{(q;q)_N (aq;q)_{(M'+M'')+N}} &= \sum_{L\ge 0} \frac{a^{L} q^{L^2 + (M'+M'')L} (q;q)_{L+M'}}{(q;q)_{N-L} (q;q)_L}\\
	&\quad\times \frac{1}{(q;q)_{L+M'} (aq;q)_{M''+(L+M')}}.
\end{align*}
Then we may as well apply \eqref{eq:trans-1} to
\begin{align*}
	\frac{1}{(q;q)_{L+M'} (aq;q)_{M''+(L+M')}}
\end{align*}
in the summand. Repeating this process $k$ times, we arrive at the following finite version of Theorem \ref{th:multi-sum-infty}.

\begin{theorem}
	For any nonnegative integers $d_1,\ldots,d_k$ and $N$,
	\begin{align}\label{eq:multi-sum-N}
		&\frac{1}{(q;q)_{N} (aq;q)_{N+d_1+\cdots+d_k}}\notag\\
		&\qquad = \sum_{n_k=0}^N \sum_{n_{k-1}=0}^{n_k+d_k}\cdots \sum_{n_1=0}^{n_2+d_2} \frac{a^{\sum_{i=1}^k n_i}q^{\sum_{i=1}^k n_i^2+\sum_{i=1}^k (d_1+\cdots+d_i)n_i}}{(q;q)_{N-n_k}(q;q)_{n_k-n_{k-1}+d_k} \cdots (q;q)_{n_2-n_1+d_2}}\notag\\
		&\qquad\quad\times \frac{(q;q)_{n_k+d_k}\cdots (q;q)_{n_2+d_2}}{(q;q)_{n_k} \cdots (q;q)_{n_1} (aq;q)_{n_1+d_1}}.
	\end{align}
\end{theorem}

Finally, we are in a position to prove Theorem \ref{th:multi-sum-infty}.

\begin{proof}[Proof of Theorem \ref{th:multi-sum-infty}]
	Let $N\to \infty$ in \eqref{eq:multi-sum-N}. Noting the fact that $1/(q;q)_n = 0$ whenever $n<0$, we may loosen the conditions of the indices in \eqref{eq:multi-sum-N} and see that
	\begin{align*}
		\lim_{N\to\infty} \RHS\eqref{eq:multi-sum-N} = \frac{1}{(q;q)_\infty} \RHS\eqref{eq:multi-sum-infty}.
	\end{align*}
	Meanwhile,
	\begin{align*}
		\lim_{N\to\infty} \LHS\eqref{eq:multi-sum-N} = \frac{1}{(q;q)_\infty (aq;q)_\infty},
	\end{align*}
	thereby implying the desired result.
\end{proof}

\section{Theorem \ref{th:S-m-infty} and its finitization}

We warm up with a proof of Theorem \ref{th:S-m-infty} by means of Theorem \ref{th:multi-sum-infty}.

\begin{proof}[Proof of Theorem \ref{th:S-m-infty}]
	We open the $q$-binomial coefficients and see that
	\begin{align*}
		\cZ_k(1,q) &= \sum_{\substack{r_k\ge \cdots\ge r_1\ge 0\\s_k\ge \cdots\ge s_1\ge 0}} \frac{q^{\sum_{i=1}^k (r_i^2-r_is_i+s_i^2)}}{(q;q)_{r_{k}-r_{k-1}}\cdots (q;q)_{r_2-r_1}(q;q)_{r_1}^2(q;q)_{s_1}}\notag\\
		&\quad\times \frac{(q;q)_{r_k-s_{k-1}}\cdots (q;q)_{r_2-s_1} (q;q)_{r_1}}{(q;q)_{r_{k}-s_{k}}\cdots(q;q)_{r_1-s_1} (q;q)_{s_k-s_{k-1}}\cdots (q;q)_{s_2-s_1} (q;q)_{s_1}}\\
		&= \sum_{\substack{r_k\ge \cdots\ge r_1\ge 0\\s_k\ge \cdots\ge s_1\ge 0}} \frac{q^{\sum_{i=1}^k (r_i^2-r_is_i+s_i^2)}}{(q;q)_{r_{k}-r_{k-1}}\cdots (q;q)_{r_2-r_1}(q;q)_{r_1}}\notag\\
		&\quad\times \frac{(q;q)_{(r_k-s_k)+(s_k-s_{k-1})}\cdots (q;q)_{(r_2-s_2)+(s_2-s_1)}}{(q;q)_{r_{k}-s_{k}}\cdots(q;q)_{r_1-s_1} (q;q)_{s_k-s_{k-1}}\cdots (q;q)_{s_2-s_1} (q;q)_{s_1}^2}.
	\end{align*}
	Now for $1\le i\le k$, we put
	\begin{align}\label{eq:di-def}
		d_i := \begin{cases}
			s_1, & i=1,\\
			s_i-s_{i-1}, & i\ge 2.
		\end{cases}
	\end{align}
	Making the change of variables for each $1\le j\le k$:
	\begin{align}\label{eq:ni-def}
		n_j := r_j - s_j,
	\end{align}
	we find that $\cZ_k(1,q)$ equals
	\begin{align*}
		&\sum_{s_1,\ldots,s_k\ge 0} \frac{q^{\sum_{i=1}^k s_i^2}}{(q;q)_{s_2-s_1}\cdots (q;q)_{s_k-s_{k-1}} (q;q)_{s_1}^2}\\
		&\times \sum_{n_1,\ldots,n_k\ge 0}\frac{q^{\sum_{i=1}^k n_i^2+\sum_{i=1}^k (d_1+\cdots+d_i)n_i}(q;q)_{n_k+d_k}\cdots (q;q)_{n_2+d_2}}{(q;q)_{n_k-n_{k-1}+d_k}\cdots (q;q)_{n_2-n_1+d_2} (q;q)_{n_k} \cdots (q;q)_{n_1} (q;q)_{n_1+d_1}},
	\end{align*}
	where we have loosened the conditions for the sums by using the vanishing of $1/(q;q)_n$ whenever $n<0$. Applying \eqref{eq:multi-sum-infty} with $a=1$ to the inner sum gives
	\begin{align*}
		\cZ_k(1,q) = \frac{1}{(q;q)_\infty}\sum_{s_1,\ldots,s_k\ge 0} \frac{q^{\sum_{i=1}^k s_i^2}}{(q;q)_{s_k-s_{k-1}}\cdots (q;q)_{s_2-s_1} (q;q)_{s_1}^2},
	\end{align*}
	which further yields \eqref{eq:S-m-infty} in view of the same reasoning.
\end{proof}

In addition, it is notable that the finite version of Theorem \ref{th:multi-sum-infty}, namely, the identity \eqref{eq:multi-sum-N}, at the same time implies a finitization of Theorem \ref{th:S-m-infty}.

\begin{theorem}\label{th:S-1-finite}
	For any nonnegative integer $N$,
	\begin{align}\label{eq:S-1-finite}
		\tcZ_k(N;1,q) = \frac{1}{(q;q)_N^3}.
	\end{align}
\end{theorem}

\begin{proof}
	Similar to how the proof of Theorem \ref{th:S-m-infty} has been proceeded, we have the simplification:
	\begin{align*}
		\tcZ_k(N;1,q) &= \sum_{s_1,\ldots,s_k\ge 0} \frac{q^{\sum_{i=1}^k s_i^2}}{(q;q)_{s_k-s_{k-1}}\cdots (q;q)_{s_2-s_1} (q;q)_{s_1}^2}\\
		&\quad\times \sum_{n_1,\ldots,n_k\ge 0} \frac{q^{\sum_{i=1}^k n_i^2+\sum_{i=1}^k (d_1+\cdots+d_i)n_i}}{(q;q)_{(N-s_k)-n_k} (q;q)_{n_k-n_{k-1}+d_k}\cdots (q;q)_{n_2-n_1+d_2}}\\
		&\quad\times \frac{(q;q)_{n_k+d_k}\cdots (q;q)_{n_2+d_2}}{(q;q)_{n_k} \cdots (q;q)_{n_1} (q;q)_{n_1+d_1}},
	\end{align*}
	where we have still used the substitutions \eqref{eq:di-def} and \eqref{eq:ni-def}. Noting that $d_1+\cdots+d_k = s_k$, we apply \eqref{eq:multi-sum-N} to simplify the inner sum over $n_1,\ldots,n_k$ as
	\begin{align*}
		\frac{1}{(q;q)_{N}(q;q)_{N-s_k}}.
	\end{align*}
	It follows that
	\begin{align*}
		\tcZ_k(N;1,q) = \frac{1}{(q;q)_N}\sum_{s_1,\ldots,s_k\ge 0} \frac{q^{\sum_{i=1}^k s_i^2}}{(q;q)_{N-s_k}(q;q)_{s_k-s_{k-1}}\cdots (q;q)_{s_2-s_1} (q;q)_{s_1}^2}.
	\end{align*}
	Applying \eqref{eq:multi-sum-N} with $d_1=\cdots = d_k = 0$ further gives
	\begin{align*}
		\tcZ_k(N;1,q) = \frac{1}{(q;q)_N} \cdot \frac{1}{(q;q)_N^2},
	\end{align*}
	which is as desired.
\end{proof}

\section{Reformulating $\tcZ_k(N;t,q)$}

To achieve the $k$-fold sum for $\tcZ_k(N;t,q)$ in \eqref{eq:Z-expression}, our first step is to reformulate it to a form that aligns with the $2k$-fold sum in \eqref{eq:strange-A2-sum}. We begin with
\begin{align*}
	\tcZ_k(N;t,q)&= \sum_{\substack{r_k\ge \cdots\ge r_1\ge 0\\s_k\ge \cdots\ge s_1\ge 0}} \frac{t^{\sum_{i=1}^k (2r_i-s_i)} q^{\sum_{i=1}^k (r_i^2-r_is_i+s_i^2)}}{(q;q)_{N-r_k}(q;q)_{r_{k}-r_{k-1}}\cdots (q;q)_{r_2-r_1}(tq;q)_{r_1}^2(q;q)_{s_1}}\notag\\
	&\quad\times \qbinom{r_k-s_{k-1}}{r_k-s_k}_q\qbinom{r_{k-1}-s_{k-2}}{r_{k-1}-s_{k-1}}_q\cdots \qbinom{r_{2}-s_1}{r_{2}-s_{2}}_q\qbinom{r_1}{r_1-s_1}_q.
\end{align*}
By opening the $q$-binomial coefficients and reorganizing the $q$-factorials, the above can be reformulated as
\begin{align*}
	\tcZ_k(N;t,q)&= \sum_{\substack{r_k\ge \cdots\ge r_1\ge 0\\s_k\ge \cdots\ge s_1\ge 0}} \frac{t^{\sum_{i=1}^k (2r_i-s_i)} q^{\sum_{i=1}^k (r_i^2-r_is_i+s_i^2)} (q;q)_{r_1}}{(q;q)_{N-r_k}(q;q)_{r_{k}-s_{k}}(q;q)_{s_k}(tq;q)_{r_1}^2(q;q)_{s_1}}\notag\\
	&\quad\times \qbinom{s_k}{s_{k-1}}_q\cdots \qbinom{s_{2}}{s_{1}}_q\qbinom{r_k-s_{k-1}}{r_{k-1}-s_{k-1}}_q\cdots \qbinom{r_{2}-s_1}{r_1-s_{1}}_q.
\end{align*}
Invoking the substitutions for $1\le i\le k$:
\begin{align*}
	n_j := r_j - s_j,
\end{align*}
we further have
\begin{align}\label{eq:Zk-middle}
	\tcZ_k(N;t,q)&= \sum_{\substack{s_1,\ldots, s_k\ge 0\\n_1,\ldots,n_k\ge 0}} \frac{t^{\sum_{i=1}^k (s_i+2n_i)} q^{\sum_{i=1}^k (s_i^2+s_in_i+n_i^2)} (q;q)_{n_1+s_1}}{(q;q)_{N-s_k-n_k}(q;q)_{s_k}(q;q)_{n_{k}}(q;q)_{s_1}(tq;q)_{n_1+s_1}^2}\notag\\
	&\quad\times \qbinom{s_k}{s_{k-1}}_q\cdots \qbinom{s_{2}}{s_{1}}_q\qbinom{n_k+s_k-s_{k-1}}{n_{k-1}}_q\cdots \qbinom{n_{2}+s_2-s_1}{n_{1}}_q.
\end{align}

Now we work on the sums over $n_1,\ldots,n_k$:
\begin{align*}
	\Sigma &:= \sum_{n_1,\ldots,n_k\ge 0} \frac{t^{\sum_{i=1}^k 2n_i} q^{\sum_{i=1}^k (n_i^2+s_in_i)} (q;q)_{n_1+s_1}}{(q;q)_{(N-s_k)-n_k}(q;q)_{n_{k}}(tq;q)_{n_1+s_1}^2}\\
	&\ \quad\times\qbinom{n_k+s_k-s_{k-1}}{n_{k-1}}_q\cdots \qbinom{n_{2}+s_2-s_1}{n_{1}}_q.
\end{align*}
Let us single out the sum over $n_1$:
\begin{align*}
	\Sigma &= \frac{1}{(q;q)_{N-s_k}} \sum_{n_2,\ldots,n_k\ge 0} t^{\sum_{i=2}^k 2n_i} q^{\sum_{i=2}^k (n_i^2+s_in_i)}\\
	&\quad\times \qbinom{N-s_k}{n_k}_q\qbinom{n_k+s_k-s_{k-1}}{n_{k-1}}_q\cdots \qbinom{n_{3}+s_3-s_2}{n_{2}}_q\\
	&\quad\times (q;q)_{n_2+s_2-s_1} \sum_{n_1\ge 0} \frac{t^{2n_1}q^{n_1^2+s_1n_1} (q;q)_{n_1+s_1}}{(q;q)_{(n_2+s_2-s_1)-n_1}(q;q)_{n_{1}}(tq;q)_{n_1+s_1}^2}.
\end{align*}
To simplify this sum over $n_1$, we require a basic hypergeometric transform.

\begin{lemma}
	For any nonnegative integers $M$ and $N$,
	\begin{align}\label{eq:trans-2}
		&\sum_{n\ge 0} \frac{a^{2n} q^{n^2 + Mn} (q;q)_{M+n}}{(q;q)_{N-n} (q;q)_n (aq;q)_{M+n}^2} = \frac{(q;q)_\infty (a^2q;q)_\infty}{(aq;q)_\infty^2 (q;q)_N} \sum_{n\ge 0} \frac{q^{(M+1)n} (a;q)_{n}^2}{(q;q)_{n} (a^2q;q)_{M+N+n}}.
	\end{align}
\end{lemma}

\begin{proof}
	We have
	\begin{align*}
		\LHS\eqref{eq:trans-2} &= \frac{(q;q)_M}{(q;q)_N(aq;q)_M^2} \sum_{n\ge 0} \frac{(-1)^n a^{2n} q^{\binom{n}{2}+(M+N+1)n} (q^{-N};q)_n (q^{M+1};q)_n}{(q;q)_n (aq^{M+1};q)_n^2}\\
		&= \frac{(q;q)_M}{(q;q)_N(aq;q)_M^2} \lim_{\tau\to 0} {}_{3}\phi_{2} \left(\begin{matrix}
			q^{-N},1/\tau,q^{M+1}\\
			aq^{M+1},aq^{M+1}
		\end{matrix};q,a^2q^{M+N+1}\tau\right)\\
		\text{\tiny (by \eqref{eq:3phi2})} &= \frac{(q;q)_M}{(q;q)_N(aq;q)_M^2} \lim_{\tau\to 0} \frac{(aq^{M+N+1},a^2q^{M+1}\tau;q)_\infty}{(aq^{M+1},a^2q^{M+N+1}\tau;q)_\infty} \\
		&\quad\times {}_{3}\phi_{2} \left(\begin{matrix}
			q^{-N},aq^{M+1}\tau,a\\
			aq^{M+1},a^2q^{M+1}\tau
		\end{matrix};q,aq^{M+N+1}\right)\\
		&= \frac{(q;q)_M}{(q;q)_N (aq;q)_M (aq;q)_{M+N}} {}_{2}\phi_{1} \left(\begin{matrix}
			q^{-N},a\\
			aq^{M+1}
		\end{matrix};q,aq^{M+N+1}\right)\\
		\text{\tiny (by \eqref{eq:Heine2})}&= \frac{(q;q)_M}{(q;q)_N (aq;q)_M (aq;q)_{M+N}} \frac{(q^{M+1},a^2q^{M+N+1};q)_\infty}{(aq^{M+1},aq^{M+N+1};q)_\infty}\\
		&\quad\times {}_{2}\phi_{1} \left(\begin{matrix}
			a,a\\
			a^2q^{M+N+1}
		\end{matrix};q,q^{M+1}\right)\\
		&= \frac{(q;q)_\infty (a^2q;q)_\infty}{(aq;q)_\infty^2 (q;q)_N (a^2q;q)_{M+N}} \sum_{n\ge 0} \frac{q^{(M+1)n} (a;q)_{n}^2}{(q;q)_{n}(a^2q^{M+N+1};q)_{n}}\\
		&= \frac{(q;q)_\infty (a^2q;q)_\infty}{(aq;q)_\infty^2 (q;q)_N} \sum_{n\ge 0} \frac{q^{(M+1)n} (a;q)_{n}^2}{(q;q)_{n} (a^2q;q)_{M+N+n}},
	\end{align*}
	as claimed.
\end{proof}

It follows by applying \eqref{eq:trans-2} to the previous sum over $n_1$ that
\begin{align*}
	\Sigma &= \frac{(q;q)_\infty (t^2q;q)_\infty}{(tq;q)_\infty^2(q;q)_{N-s_k}} \sum_{n_2,\ldots,n_k\ge 0} t^{\sum_{i=2}^k 2n_i}  q^{\sum_{i=2}^k (n_i^2+s_in_i)} \\
	&\quad\times \qbinom{N-s_k}{n_k}_q \qbinom{n_k+s_k-s_{k-1}}{n_{k-1}}_q\cdots \qbinom{n_{3}+s_3-s_2}{n_{2}}_q\\
	&\quad\times \sum_{n_1\ge 0} \frac{q^{n_1+s_1n_1} (t;q)_{n_1}^2}{(q;q)_{n_{1}}(t^2q;q)_{n_1+n_2+s_2}}.
\end{align*}
Interchanging the sum over $n_1$ and the remaining sums gives
\begin{align*}
	\Sigma &= \frac{(q;q)_\infty (t^2q;q)_\infty}{(tq;q)_\infty^2(q;q)_{N-s_k}} \sum_{n_1\ge 0} \frac{q^{n_1+s_1n_1} (t;q)_{n_1}^2}{(q;q)_{n_{1}}} \sum_{n_2,\ldots,n_k\ge 0} \frac{t^{\sum_{i=2}^k 2n_i} q^{\sum_{i=2}^k (n_i^2+s_in_i)}}{(t^2q;q)_{n_1+n_2+s_2}}\\
	&\quad\times  \qbinom{N-s_k}{n_k}_q\qbinom{n_k+s_k-s_{k-1}}{n_{k-1}}_q\cdots \qbinom{n_{3}+s_3-s_2}{n_{2}}_q.
\end{align*}

Our next trick relies on a slight extension of a transform due to Warnaar \cite[p.~746, Lemma 7.2]{War2023}.

\begin{lemma}
	Let $m_0$ be a nonnegative integer and let $u_1\le u_2\le \cdots \le u_{k+1}$ be integers. We have, for any $\ell\in \{0,1,\ldots,k\}$,
	\begin{align}
		&\sum_{m_1,\ldots,m_k\ge 0} \frac{t^{\sum_{i=1}^k m_i}q^{\sum_{i=1}^k m_i(m_i+u_i)}}{(tq;q)_{m_k+u_{k+1}}} \prod_{i=1}^k \qbinom{m_{i-1}}{m_i}_q\notag\\
		&\quad = \sum_{m_1,\ldots,m_k\ge 0} \frac{t^{\sum_{i=1}^k m_i}q^{\sum_{i=1}^k m_i(m_i+u_i)}}{(tq;q)_{m_\ell+m_{\ell+1}+u_{\ell+1}}} \prod_{i=1}^\ell \qbinom{m_{i-1}}{m_i}_q \prod_{i=\ell+1}^k \qbinom{m_{i+1}+u_{i+1}-u_i}{m_i}_q,
	\end{align}
	where $m_{k+1}:=0$.
\end{lemma}

\begin{proof}
	The proof is almost identical to that for \cite[p.~746, Lemma 7.2]{War2023}. The only modification is that in the following identity \cite[p.~746, above eq.~(7.6)]{War2023}:
	\begin{align*}
		\sum_{k\ge 0} (-z)^k q^{\binom{k}{2}} \frac{(a;q)_k (cq^k;q)_\infty}{(q;q)_k} = \sum_{k\ge 0} (-c)^k q^{\binom{k}{2}} \frac{(az/c;q)_k (zq^k;q)_\infty}{(q;q)_k},
	\end{align*}
	we instead set $(a,c,z)\mapsto (q^{-(n_2-p)},tq^{n_1+1},tq^{n_2+1})$ so as to extend \cite[p.~747, eq.~(7.7)]{War2023} as
	\begin{align*}
		\sum_{m\ge 0} \frac{t^mq^{m(m+p)}}{(tq;q)_{m+n_1}} \qbinom{n_2-p}{m}_q = \sum_{m\ge 0} \frac{t^mq^{m(m+p)}}{(tq;q)_{m+n_2}} \qbinom{n_1-p}{m}_q.
	\end{align*}
	The rest follows by the same induction argument.
\end{proof}

The above lemma tells us that
\begin{align*}
	&\sum_{n_2,\ldots,n_k\ge 0} \frac{t^{\sum_{i=2}^k 2n_i} q^{\sum_{i=2}^k (n_i^2+s_in_i)}}{(t^2q;q)_{n_1+n_2+s_2}} \qbinom{N-s_k}{n_k}_q\qbinom{n_k+s_k-s_{k-1}}{n_{k-1}}_q\cdots \qbinom{n_{3}+s_3-s_2}{n_{2}}_q\\
	&\qquad = \sum_{n_2,\ldots,n_k\ge 0} \frac{t^{\sum_{i=2}^k 2n_i} q^{\sum_{i=2}^k (n_i^2+s_in_i)}}{(t^2q;q)_{N+n_k}} \qbinom{n_1}{n_2}_q \cdots \qbinom{n_{k-1}}{n_k}_q.
\end{align*}
As a consequence,
\begin{align*}
	\Sigma &= \frac{(q;q)_\infty (t^2q;q)_\infty}{(tq;q)_\infty^2(q;q)_{N-s_k}} \sum_{n_1,\ldots,n_k\ge 0} t^{\sum_{i=2}^k 2n_i} q^{n_1+s_1n_1+\sum_{i=2}^k (n_i^2+s_in_i)}\\
	&\quad\times\frac{(t;q)_{n_1}^2}{(q;q)_{n_{1}} (t^2q;q)_{N+n_k}}\qbinom{n_1}{n_2}_q \cdots \qbinom{n_{k-1}}{n_k}_q.
\end{align*}

Finally, substituting the above into \eqref{eq:Zk-middle} and renaming $s_i$ by $m_i$, we are led to the following reformulation of $\tcZ_k(N;t,q)$.

\begin{theorem}
	For any nonnegative integer $N$,
	\begin{align}\label{eq:Z-new}
		\tcZ_k(N;t,q)&= \frac{(q;q)_\infty(t^2q;q)_\infty}{(tq;q)_\infty^2}\notag\\
		&\quad\times \sum_{\substack{m_1,\ldots,m_k\ge 0\\n_1,\ldots,n_k\ge 0}} \frac{t^{-2n_1+\sum_{i=1}^k (m_i+2n_i)} q^{-n_1^2+n_1+\sum_{i=1}^k (m_i^2+m_in_i+n_i^2)} (t;q)_{n_1}^2}{(q;q)_{N-m_k}(t^2q;q)_{N+n_{k}}(q;q)_{m_k}(q;q)_{m_1}(q;q)_{n_1}}\notag\\
		&\quad\times \qbinom{m_k}{m_{k-1}}_q \qbinom{m_{k-1}}{m_{k-2}}_q\cdots \qbinom{m_{2}}{m_{1}}_q \qbinom{n_1}{n_2}_q \cdots \qbinom{n_{k-2}}{n_{k-1}}_q\qbinom{n_{k-1}}{n_k}_q.
	\end{align}
\end{theorem}

\section{A semi-truncation}

We move on to the following multisum:
\begin{align}
	\cV_k(N;t,q)&:= \sum_{\substack{m_1,\ldots,m_k\ge 0\\n_1,\ldots,n_k\ge 0}} \frac{t^{-2n_1+\sum_{i=1}^k (m_i+2n_i)} q^{-n_1^2+n_1+\sum_{i=1}^k (m_i^2+m_in_i+n_i^2)} (t;q)_{n_1}^2}{(q;q)_{N-m_k}(q;q)_{m_k}(q;q)_{m_1}(q;q)_{n_1}}\notag\\
	&\ \quad\times \qbinom{m_k}{m_{k-1}}_q \qbinom{m_{k-1}}{m_{k-2}}_q\cdots \qbinom{m_{2}}{m_{1}}_q \qbinom{n_1}{n_2}_q \cdots \qbinom{n_{k-2}}{n_{k-1}}_q\qbinom{n_{k-1}}{n_k}_q.
\end{align}
It is notable that only the sums over $m_1,\ldots,m_k$ are finite.

Let us assume that $k\ge 2$.

We start by opening the $q$-binomial coefficients:
\begin{align*}
	\cV_k(N;t,q)&= \sum_{\substack{m_1,\ldots,m_k\ge 0\\n_1,\ldots,n_k\ge 0}} \frac{t^{-2n_1+\sum_{i=1}^k (m_i+2n_i)} q^{-n_1^2+n_1+\sum_{i=1}^k (m_i^2+m_in_i+n_i^2)} (t;q)_{n_1}^2}{(q;q)_{N-m_k}(q;q)_{n_k}(q;q)_{m_1}^2}\notag\\
	&\quad\times \frac{1}{(q;q)_{m_k-m_{k-1}}\cdots (q;q)_{m_2-m_1} (q;q)_{n_1-n_2}\cdots (q;q)_{n_{k-1}-n_k}}.
\end{align*}
Singling out the sums over $m_1,\ldots,m_{k-1}$ and $n_1,\ldots,n_{k-1}$ then gives
\begin{align*}
	\cV_k(N;t,q) &= \sum_{m_k,n_k\ge 0} \frac{t^{m_k+2n_k}q^{m_k^2+m_kn_k+n_k^2}}{(q;q)_{N-m_k}(q;q)_{n_k}}\\
	&\quad\times \sum_{\substack{m_1,\ldots,m_{k-1}\ge 0\\n_1,\ldots,n_{k-1}\ge n_k}} \frac{t^{-2n_1+\sum_{i=1}^{k-1} (m_i+2n_i)} q^{-n_1^2+n_1+\sum_{i=1}^{k-1} (m_i^2+m_in_i+n_i^2)} (t;q)_{n_1}^2}{(q;q)_{m_k-m_{k-1}}(q;q)_{n_{k-1}-n_k}(q;q)_{m_1}^2}\notag\\
	&\quad\times \frac{1}{(q;q)_{m_{k-1}-m_{k-2}}\cdots (q;q)_{m_2-m_1} (q;q)_{n_1-n_2}\cdots (q;q)_{n_{k-2}-n_{k-1}}}.
\end{align*}
Now we make the substitutions for $1\le i\le k-1$:
\begin{align*}
	n_i \mapsto n_i+n_k.
\end{align*}
Then,
\begin{align*}
	&\cV_k(N;t,q)\\
	&\quad = \sum_{m_k,n_k\ge 0} \frac{t^{m_k+2(k-1)n_k}q^{m_k^2+(m_k+1)n_k+(k-1)n_k^2}(t;q)_{n_k}^2}{(q;q)_{N-m_k}(q;q)_{n_k}}\\
	&\quad\quad\times \sum_{\substack{m_1,\ldots,m_{k-1}\ge 0\\n_1,\ldots,n_{k-1}\ge 0}} \frac{(tq^{n_k})^{-2n_1+\sum_{i=1}^{k-1} (m_i+2n_i)} q^{-n_1^2+n_1+\sum_{i=1}^{k-1} (m_i^2+m_in_i+n_i^2)} (tq^{n_k};q)_{n_1}^2}{(q;q)_{m_k-m_{k-1}}(q;q)_{n_{k-1}}(q;q)_{m_1}^2}\notag\\
	&\quad\quad\times \frac{1}{(q;q)_{m_{k-1}-m_{k-2}}\cdots (q;q)_{m_2-m_1} (q;q)_{n_1-n_2}\cdots (q;q)_{n_{k-2}-n_{k-1}}}.
\end{align*}
In other words,
\begin{align}\label{eq:V-rec}
	\cV_k(N;t,q) = \sum_{m,n\ge 0} \frac{t^{m+2(k-1)n}q^{m^2+(m+1)n+(k-1)n^2}(t;q)_{n}^2}{(q;q)_{N-m}(q;q)_{n}} \cV_{k-1}(m;tq^{n},q).
\end{align}

Now we simplify $\cV_k(N;t,q)$ to a great extent as follows.

\begin{theorem}
	For any nonnegative integer $N$,
	\begin{align}\label{eq:V-expression}
		\cV_k(N;t,q) &= \frac{(tq;q)_\infty}{(q;q)_\infty (q;q)_N} \sum_{n_1,\ldots,n_k\ge 0} (-1)^{n_k} t^{-n_k+\sum_{i=1}^k 2n_i} q^{-\binom{n_k}{2}+\sum_{i=1}^k n_i^2}\notag\\
		&\quad\times \frac{(t;q)_{n_k}}{(q;q)_{n_k}(tq;q)_{N+n_1}} \qbinom{n_k}{n_{k-1}}_q\qbinom{n_{k-1}}{n_{k-2}}_q\cdots \qbinom{n_2}{n_1}_q.
	\end{align}
\end{theorem}

Our strategy is to apply induction on $k$ by means of \eqref{eq:V-rec}. Here we work on the base case at $k=1$ and the inductive step separately.

\begin{proof}[Proof of the base case]
	Recall that
	\begin{align*}
		\cV_1(N;t,q) = \sum_{m_1,n_1\ge 0} \frac{t^{m_1} q^{m_1^2 + m_1n_1 + n_1} (t;q)_{n_1}^2}{(q;q)_{N-m_1} (q;q)_{m_1}^2 (q;q)_{n_1}}.
	\end{align*}
	We first focus on the sum over $m_1$:
	\begin{align*}
		\cV_1(N;t,q) &= \sum_{n_1\ge 0} \frac{q^{n_1} (t;q)_{n_1}^2}{(q;q)_{n_1}} \sum_{m_1\ge 0} \frac{t^{m_1} q^{m_1^2 + n_1m_1}}{(q;q)_{N-m_1} (q;q)_{m_1}^2}.
	\end{align*}
	It is clear that
	\begin{align*}
		\sum_{m_1\ge 0} \frac{t^{m_1} q^{m_1^2 + m_1n_1}}{(q;q)_{N-m_1} (q;q)_{m_1}^2} &= \frac{1}{(q;q)_N} \sum_{m_1\ge 0} \frac{(-1)^{m_1} t^{m_1} q^{\binom{m_1}{2}+(N+n_1+1)m_1} (q^{-N};q)_{m_1}}{(q;q)_{n_1}^2}\\
		&= \frac{1}{(q;q)_N} \lim_{\tau\to 0} {}_{2}\phi_{1} \left(\begin{matrix}
			1/\tau,q^{-N}\\
			q
		\end{matrix};q,tq^{N+n_1+1}\tau\right).
	\end{align*}
	We temporally assume that $|t|<1$ to ensure the convergence condition for the application of Heine's third transformation \eqref{eq:Heine3} to the ${}_{2}\phi_{1}$ series especially when $n_1 = 0$. Then,
	\begin{align*}
		\sum_{m_1\ge 0} \frac{t^{m_1} q^{m_1^2 + m_1n_1}}{(q;q)_{N-m_1} (q;q)_{m_1}^2} &= \frac{1}{(q;q)_N}\cdot (tq^{n_1};q)_\infty {}_{2}\phi_{1} \left(\begin{matrix}
			0,q^{N+1}\\
			q
		\end{matrix};q,tq^{n_1}\tau\right)\\
		&= \frac{(t;q)_\infty}{(q;q)_N (t;q)_{n_1}} \sum_{m_1\ge 0} \frac{t^{m_1}q^{n_1m_1}(q^{N+1};q)_{m_1}}{(q;q)_{m_1}^2}.
	\end{align*}
	It follows that
	\begin{align*}
		\cV_1(N;t,q) &= \frac{(t;q)_\infty}{(q;q)_N} \sum_{m_1\ge 0} \frac{t^{m_1} (q^{N+1};q)_{m_1}}{(q;q)_{m_1}^2} \sum_{n_1\ge 0} \frac{q^{(m_1+1)n_1} (t;q)_{n_1}}{(q;q)_{n_1}}\\
		\text{\tiny (by \eqref{eq:q-binomial})}&= \frac{(t;q)_\infty}{(q;q)_N} \sum_{m_1\ge 0} \frac{t^{m_1} (q^{N+1};q)_{m_1}}{(q;q)_{m_1}^2} \frac{(tq^{m_1+1};q)_\infty}{(q^{m_1+1};q)_\infty}\\
		&= \frac{(t;q)_\infty (tq;q)_\infty}{(q;q)_\infty (q;q)_N} \sum_{m_1\ge 0} \frac{t^{m_1} (q^{N+1};q)_{m_1}}{(q;q)_{m_1} (tq;q)_{m_1}}\\
		&= \frac{(t;q)_\infty (tq;q)_\infty}{(q;q)_\infty (q;q)_N} \lim_{\tau\to 0} {}_{2}\phi_{1} \left(\begin{matrix}
			q^{N+1},tq\tau\\
			tq
		\end{matrix};q,t\right)\\
		\text{\tiny (by \eqref{eq:Heine1})} &= \frac{(t;q)_\infty (tq;q)_\infty}{(q;q)_\infty (q;q)_N} \lim_{\tau\to 0} \frac{(tq\tau,tq^{N+1};q)_\infty}{(tq,t;q)_\infty} {}_{2}\phi_{1} \left(\begin{matrix}
			1/\tau,t\\
			tq^{N+1}
		\end{matrix};q,tq\tau\right)\\
		&= \frac{(tq;q)_\infty}{(q;q)_\infty (q;q)_N} \sum_{n_1\ge 0} \frac{(-1)^{n_1} t^{n_1} q^{\binom{n_1+1}{2}} (t;q)_{n_1}}{(q;q)_{n_1} (tq;q)_{N+n_1}}.
	\end{align*}
	It is notable that this relation can be analytically continued from $|t|<1$, which has been assumed earlier, to $t\in \mathbb{C}$. Hence, we arrive at \eqref{eq:V-expression} for $k=1$.
\end{proof}

\begin{proof}[Proof of the inductive step]
	Assume that \eqref{eq:V-expression} is valid for some $k-1$ with $k\ge 2$. Thus,
	\begin{align*}
		\cV_{k-1}(m;tq^{n},q) &= \frac{(tq^{n+1};q)_\infty}{(q;q)_\infty (q;q)_m} \sum_{n_1,\ldots,n_{k-1}\ge 0} (-1)^{n_{k-1}} t^{-n_{k-1}+\sum_{i=1}^{k-1} 2n_i}\notag\\
		&\quad\times \frac{q^{-\binom{n_{k-1}}{2}-nn_{k-1}+\sum_{i=1}^{k-1} (n_i^2+2nn_i)} (tq^n;q)_{n_{k-1}}}{(q;q)_{n_{k-1}-n_{k-2}}\cdots(q;q)_{n_2-n_1}(q;q)_{n_1}(tq^{n+1};q)_{m+n_1}}.
	\end{align*}
	Invoking \eqref{eq:V-rec},
	\begin{align*}
		\cV_k(N;t,q) &= \sum_{m,n\ge 0} \frac{(-1)^n t^{m+n}q^{m^2+mn+\binom{n+1}{2}}(t;q)_{n}^2}{(q;q)_{N-m}(q;q)_{n}} \frac{(tq^{n+1};q)_\infty}{(q;q)_\infty (q;q)_m}\\
		&\quad\times \sum_{n_1,\ldots,n_{k-1}\ge 0} (-1)^{n_{k-1}+n} t^{-(n_{k-1}+n)+\sum_{i=1}^{k-1} 2(n_i+n)}\notag\\
		&\quad\times \frac{q^{-\binom{n_{k-1}+n}{2}+\sum_{i=1}^{k-1} (n_i+n)^2} (tq^n;q)_{n_{k-1}}}{(q;q)_{n_{k-1}-n_{k-2}}\cdots(q;q)_{n_2-n_1}(q;q)_{n_1}(tq^{n+1};q)_{m+n_1}}\\
		&= \frac{(tq;q)_\infty}{(q;q)_\infty} \sum_{m,n\ge 0} \frac{(-1)^n t^{m+n}q^{m^2+mn+\binom{n+1}{2}}(t;q)_{n}}{(q;q)_{N-m} (q;q)_m(q;q)_{n}}\\
		&\quad\times \sum_{n_1,\ldots,n_{k-1}\ge 0} (-1)^{n_{k-1}+n} t^{-(n_{k-1}+n)+\sum_{i=1}^{k-1} 2(n_i+n)}\notag\\
		&\quad\times \frac{q^{-\binom{n_{k-1}+n}{2}+\sum_{i=1}^{k-1} (n_i+n)^2} (t;q)_{n_{k-1}+n}}{(q;q)_{n_{k-1}-n_{k-2}}\cdots(q;q)_{n_2-n_1}(q;q)_{n_1}(tq;q)_{m+(n_1+n)}}.
	\end{align*}
	We put, for each $1\le i\le k-1$:
	\begin{align*}
		l_i := n_i+n,
	\end{align*}
	and interchange the sums over $m,n$ and the rest. Then,
	\begin{align*}
		\cV_k(N;t,q) &= \frac{(tq;q)_\infty}{(q;q)_\infty} \sum_{l_1,\ldots,l_{k-1}\ge 0} \frac{(-1)^{l_{k-1}} t^{-l_{k-1}+\sum_{i=1}^{k-1} 2l_i} q^{-\binom{l_{k-1}}{2}+\sum_{i=1}^{k-1} l_i^2}(t;q)_{l_{k-1}}}{(q;q)_{l_{k-1}-l_{k-2}}\cdots(q;q)_{l_2-l_1}}\notag\\
		&\quad\times \sum_{m,n\ge 0} \frac{(-1)^n t^{m+n}q^{m^2+mn+\binom{n+1}{2}}(t;q)_{n}}{(q;q)_{N-m} (q;q)_{l_1-n}(tq;q)_{l_1+m} (q;q)_m(q;q)_{n}}.
	\end{align*}
	Hence, as long as we can show
	\begin{align}\label{eq:induction-key}
		&\sum_{m,n\ge 0} \frac{(-1)^n t^{m+n}q^{m^2+mn+\binom{n+1}{2}}(t;q)_{n}}{(q;q)_{N-m} (q;q)_{l_1-n}(tq;q)_{l_1+m} (q;q)_m(q;q)_{n}}\notag\\
		&\qquad\qquad\qquad\qquad\qquad\qquad = \frac{1}{(q;q)_N} \sum_{l_0\ge 0} \frac{t^{2l_0}q^{l_0^2}}{(q;q)_{l_1-l_0} (q;q)_{l_0} (tq;q)_{N+l_0}},
	\end{align}
	then \eqref{eq:V-expression} holds for $k$ by renaming the indices with $l_i\mapsto n_{i+1}$ for each $0\le i\le k-1$. To acquire this last ingredient in the recipe, we single out the sum over $n$:
	\begin{align*}
		\LHS\eqref{eq:induction-key} = \sum_{m\ge 0} \frac{t^{m}q^{m^2}}{(q;q)_{N-m} (q;q)_m(tq;q)_{l_1+m}} \sum_{n\ge 0} \frac{(-1)^n t^{n}q^{\binom{n}{2}+(m+1)n}(t;q)_{n}}{(q;q)_{l_1-n}(q;q)_{n}}.
	\end{align*}
	Note that
	\begin{align*}
		&\sum_{n\ge 0} \frac{(-1)^n t^{n}q^{\binom{n}{2}+(m+1)n}(t;q)_{n}}{(q;q)_{l_1-n}(q;q)_{n}}\\
		&\qquad = \frac{1}{(q;q)_{l_1}} \lim_{\tau\to 0} {}_{2}\phi_{1} \left(\begin{matrix}
			t,q^{-l_1}\\
			t^2q^{m+1}\tau
		\end{matrix};q,tq^{l_1+m+1}\right)\\
		&\qquad = \frac{1}{(q;q)_{l_1}} \lim_{\tau\to 0} \frac{(t^2q^{l_1+m+1}\tau,tq^{m+1};q)_\infty}{(t^2q^{m+1}\tau,tq^{l_1+m+1};q)_\infty} {}_{2}\phi_{1} \left(\begin{matrix}
			1/\tau,q^{-l_1}\\
			tq^{m+1}
		\end{matrix};q,t^2q^{l_1+m+1}\tau\right)\\
		&\qquad = \frac{(tq^{m+1};q)_\infty}{(tq^{l_1+m+1};q)_\infty} \sum_{l_0\ge 0} \frac{t^{2l_0}q^{l_0^2+ml_0}}{(q;q)_{l_1-l_0}(q;q)_{l_0}(tq^{m+1};q)_{l_0}},
	\end{align*}
	where we have applied Heine's second transform \eqref{eq:Heine2}. Hence,
	\begin{align*}
		&\LHS\eqref{eq:induction-key}\\
		&= \sum_{l_0\ge 0} \frac{t^{2l_0}q^{l_0^2}}{(q;q)_{l_1-l_0}(q;q)_{l_0}(tq;q)_{l_0}} \sum_{m\ge 0} \frac{t^{m}q^{m^2+l_0m}}{(q;q)_{N-m} (q;q)_m(tq^{l_0+1};q)_{m}}\\
		&= \frac{1}{(q;q)_N}\sum_{l_0\ge 0} \frac{t^{2l_0}q^{l_0^2}}{(q;q)_{l_1-l_0}(q;q)_{l_0}(tq;q)_{l_0}} \lim_{\tau\to 0} {}_{2}\phi_{1} \left(\begin{matrix}
			1/\tau,q^{-N}\\
			tq^{l_0+1}
		\end{matrix};q,tq^{N+l_0+1}\tau\right).
	\end{align*}
	Applying the first $q$-Chu--Vandermonde sum \eqref{eq:qCV-1} yields
	\begin{align*}
		\LHS\eqref{eq:induction-key} = \frac{1}{(q;q)_N}\sum_{l_0\ge 0} \frac{t^{2l_0}q^{l_0^2}}{(q;q)_{l_1-l_0}(q;q)_{l_0}(tq;q)_{l_0}} \frac{1}{(tq^{l_0+1};q)_N},
	\end{align*}
	which is exactly what we need.
\end{proof}

\section{$q$-Lebesgue identity}

Recall that
\begin{align*}
	\tcZ_1(N;t,q) = \frac{(q;q)_\infty(t^2q;q)_\infty}{(tq;q)_\infty^2} \sum_{m_1,n_1\ge 0} \frac{t^{m_1} q^{m_1^2 + m_1n_1 + n_1} (t;q)_{n_1}^2}{(q;q)_{N-m_1} (t^2q;q)_{N+n_1} (q;q)_{m_1}^2 (q;q)_{n_1}}.
\end{align*}

Unlike how we treat $\cV_1(N;t,q)$, this time we focus on the sum over $n_1$ at first:
\begin{align*}
	\tcZ_1(N;t,q) &= \frac{(q;q)_\infty(t^2q;q)_\infty}{(tq;q)_\infty^2 (t^2q;q)_N} \sum_{m_1\ge 0} \frac{t^{m_1} q^{m_1^2}}{(q;q)_{N-m_1} (q;q)_{m_1}^2}\\
	&\quad\times \sum_{n_1\ge 0} \frac{q^{(m_1+1)n_1} (t;q)_{n_1}^2}{(q;q)_{n_1} (t^2q^{N+1};q)_{n_1}}.
\end{align*}
Note that
\begin{align*}
	\sum_{n_1\ge 0} \frac{q^{(m_1+1)n_1} (t;q)_{n_1}^2}{(q;q)_{n_1} (t^2q^{N+1};q)_{n_1}} &= {}_{2}\phi_{1} \left(\begin{matrix}
		t,t\\
		t^2q^{N+1}
	\end{matrix};q,q^{m_1+1}\right)\\
	\text{\tiny (by \eqref{eq:Heine2})} &= \frac{(tq^{N+1},tq^{m_1+1};q)_\infty}{(t^2q^{N+1},q^{m_1+1};q)_\infty} {}_{2}\phi_{1} \left(\begin{matrix}
		q^{m_1-N},t\\
		tq^{m_1+1}
	\end{matrix};q,tq^{N+1}\right)\\
	&= \frac{(tq^{N+1},tq^{m_1+1};q)_\infty}{(t^2q^{N+1},q^{m_1+1};q)_\infty} \sum_{n_1\ge 0} \frac{t^{n_1} q^{(N+1)n_1} (q^{m_1-N},t;q)_{n_1}}{(q,tq^{m_1+1};q)_{n_1}}.
\end{align*}
Therefore,
\begin{align*}
	\tcZ_1(N;t,q) &= \frac{1}{(tq;q)_N} \sum_{m_1,n_1\ge 0} \frac{t^{m_1+n_1} q^{m_1^2+(N+1)n_1} (t;q)_{n_1} (q^{m_1-N};q)_{n_1}}{(q;q)_{N-m_1} (q;q)_{m_1} (q;q)_{n_1} (tq;q)_{m_1+n_1}}\\
	&= \frac{1}{(tq;q)_N} \sum_{n_1\ge 0} \frac{(-1)^{n_1} t^{n_1} q^{\binom{n_1+1}{2}} (t;q)_{n_1}}{(q;q)_{n_1} (tq;q)_{n_1}}\\
	&\quad\times \sum_{m_1\ge 0} \frac{t^{m_1} q^{m_1^2 + n_1 m_1}}{(q;q)_{(N-n_1)-m_1} (q;q)_{m_1} (tq^{n_1+1};q)_{m_1}}.
\end{align*}
For the inner sum over $m_1$, we have
\begin{align*}
	&\!\!\!\!\!\!\!\!\!\!\!\!\!\sum_{m_1\ge 0} \frac{t^{m_1} q^{m_1^2+n_1m_1}}{(q;q)_{(N-n_1)-m_1}(q;q)_{m_1}(tq^{n_1+1};q)_{m_1}}\\
	&= \frac{1}{(q;q)_{N-n_1}}\lim_{\tau\to 0} {}_{2}\phi_{1} \left(\begin{matrix}
		1/\tau,q^{-N+n_1}\\
		tq^{n_1+1}
	\end{matrix};q,tq^{N+1}\tau\right)\\
	\text{\tiny (by \eqref{eq:qGauss})}& = \frac{1}{(q;q)_{N-n_1}} \frac{(tq^{N+1};q)_\infty}{(tq^{n_1+1};q)_\infty}.
\end{align*}
It follows that
\begin{align}\label{eq:Z1-1}
	\tcZ_1(N;t,q) = \frac{1}{(tq;q)_N^2}\sum_{n_1\ge 0} \frac{(-1)^{n_1} t^{n_1} q^{\binom{n_1+1}{2}} (t;q)_{n_1}}{(q;q)_{N-n_1} (q;q)_{n_1}}.
\end{align}

We may further rewrite the above as
\begin{align*}
	\tcZ_1(N;t,q) &= \frac{1}{(q;q)_N (tq;q)_N^2} \sum_{n_1\ge 0} \frac{t^{n_1}q^{(N+1)n_1} (t;q)_{n_1} (q^{-N};q)_{n_1}}{(q;q)_{n_1}}\\
	& = \frac{1}{(q;q)_N (tq;q)_N^2} \lim_{\tau\to 0} {}_{2}\phi_{1} \left(\begin{matrix}
		t,q^{-N}\\
		t^2q \tau
	\end{matrix};q,tq^{N+1}\right)\\
	\text{\tiny (by \eqref{eq:Heine2})} &= \frac{1}{(q;q)_N (tq;q)_N^2} \lim_{\tau\to 0} \frac{(t^2q^{N+1}\tau,tq;q)_\infty}{(t^2q \tau,tq^{N+1};q)_\infty} {}_{2}\phi_{1} \left(\begin{matrix}
		1/\tau,q^{-N}\\
		tq
	\end{matrix};q,t^2q^{N+1}\tau\right)\\
	&= \frac{1}{(q;q)_N (tq;q)_N} \sum_{n_1\ge 0} \frac{(-1)^{n_1}t^{2n_1}q^{\binom{n_1}{2}+(N+1)n_1} (q^{-N};q)_{n_1}}{(q;q)_{n_1} (tq;q)_{n_1}}. 
\end{align*}
Consequently,
\begin{align}\label{eq:Z1-2}
	\tcZ_1(N;t,q) = \frac{1}{(tq;q)_N}\sum_{n_1\ge 0} \frac{t^{2n_1} q^{n_1^2}}{(q;q)_{N-n_1} (q;q)_{n_1} (tq;q)_{n_1}}.
\end{align}

Now recall a polynomial identity discovered by Paule \cite[p.~272, eq.~(43)]{Pau1985}:
\begin{align*}
	\sum_{n=-N}^N (-1)^n q^{2n^2} \qbinom{2N}{N-n}_q = \frac{(q;q)_{2N}}{(q;q)_N} \sum_{n=0}^N \frac{q^{n^2}}{(-q;q)_n} \qbinom{N}{n}_q.
\end{align*}
Invoking \eqref{eq:Z1-1} and \eqref{eq:Z1-2} with $t=-1$, we have the following identity.

\begin{theorem}
	For any nonnegative integer $N$,
	\begin{align}\label{eq:Lebe-finite}
		\sum_{n\ge 0} q^{\binom{n+1}{2}} (-1;q)_{n} \qbinom{N}{n}_q = \frac{1}{(q;q^2)_N} \sum_{n=-N}^N (-1)^n q^{2n^2} \qbinom{2N}{N-n}_q.
	\end{align}
\end{theorem}

Remarkably, the above serves as a new finitization of a special case of the \emph{$q$-Lebesgue sum} \cite[p.~21, Corollary 2.7 with $a=-1$]{And1998}:
\begin{align*}
	\sum_{n\ge 0}\frac{q^{\binom{n+1}{2}} (-1;q)_n}{(q;q)_n} = (-q;q)_\infty (-q;q^2)_\infty.
\end{align*}
Another finitization of this identity was discovered by Santos and Sills \cite[p.~128, eq.~(3.1)]{SS2002}, while for the generic $q$-Lebesgue sum, we have witnessed finite analogs derived by Alladi and Berkovich \cite[p.~803, eq.~(1.15)]{AB2004} and Rowell \cite[p.~786, eq.~(1.5)]{Row2010}.

\section{Toward the $\mathrm{A}_1$-type sum in Theorem \ref{th:Z-expression}}

Our objective in this part is to reduce $\tcZ_k(N;t,q)$ to the $\mathrm{A}_1$-type sum as recorded in Theorem \ref{th:Z-expression}. We start with the following result.

\begin{theorem}\label{th:Z-expression-new}
	When $k\ge 2$, for any nonnegative integer $N$,
	\begin{align}\label{eq:Z-expression-new}
		&\tcZ_k(N;t,q)\notag\\
		&\qquad= \frac{(t^2q;q)_\infty}{(tq;q)_\infty(q;q)_N(tq;q)_N} \sum_{n_1,\ldots,n_k\ge 0} (-1)^{n_k} t^{-n_k+\sum_{i=1}^k 2n_i} q^{-\binom{n_k}{2}+\sum_{i=1}^k n_i^2}\notag\\
		&\qquad\quad\times \frac{(t;q)_{n_k}}{(q;q)_{n_k}(t^2q;q)_{N+n_1}} \qbinom{n_k}{n_{k-1}}_q\qbinom{n_{k-1}}{n_{k-2}}_q\cdots \qbinom{n_2}{n_1}_q.
	\end{align}
\end{theorem}

\begin{proof}
	We recall \eqref{eq:Z-new} and mimic how \eqref{eq:V-rec} is derived so as to get the relation:
	\begin{align}\label{eq:Z-rec}
		\tcZ_k(N;t,q) &= \frac{(q;q)_\infty(t^2q;q)_\infty}{(tq;q)_\infty^2}\sum_{m,n\ge 0} t^{m+2(k-1)n}q^{m^2+(m+1)n+(k-1)n^2}\notag\\
		&\quad\times \frac{(t;q)_{n}^2}{(q;q)_{N-m}(t^2q;q)_{N+n}(q;q)_{n}} \cV_{k-1}(m;tq^{n},q).
	\end{align}
	Now the term $\cV_{k-1}(m;tq^{n},q)$ can be replaced by means of \eqref{eq:V-expression}. Then,
	\begin{align*}
		\tcZ_k(N;t,q) &= \frac{(t^2q;q)_\infty}{(tq;q)_\infty} \sum_{m,n\ge 0} \frac{(-1)^n t^{m+n}q^{m^2+mn+\binom{n+1}{2}}(t;q)_{n}}{(q;q)_{N-m}(t^2q;q)_{N+n} (q;q)_m(q;q)_{n}}\\
		&\quad\times \sum_{n_1,\ldots,n_{k-1}\ge 0} (-1)^{n_{k-1}+n} t^{-(n_{k-1}+n)+\sum_{i=1}^{k-1} 2(n_i+n)}\notag\\
		&\quad\times \frac{q^{-\binom{n_{k-1}+n}{2}+\sum_{i=1}^{k-1} (n_i+n)^2} (t;q)_{n_{k-1}+n}}{(q;q)_{n_{k-1}-n_{k-2}}\cdots(q;q)_{n_2-n_1}(q;q)_{n_1}(tq;q)_{m+(n_1+n)}}\\
		&= \frac{(t^2q;q)_\infty}{(tq;q)_\infty} \sum_{l_1,\ldots,l_{k-1}\ge 0} \frac{(-1)^{l_{k-1}} t^{-l_{k-1}+\sum_{i=1}^{k-1} 2l_i} q^{-\binom{l_{k-1}}{2}+\sum_{i=1}^{k-1} l_i^2}(t;q)_{l_{k-1}}}{(q;q)_{l_{k-1}-l_{k-2}}\cdots(q;q)_{l_2-l_1}}\notag\\
		&\quad\times \sum_{m,n\ge 0} \frac{(-1)^n t^{m+n}q^{m^2+mn+\binom{n+1}{2}}(t;q)_{n}}{(q;q)_{N-m}(t^2q;q)_{N+n} (q;q)_{l_1-n}(tq;q)_{l_1+m} (q;q)_m(q;q)_{n}}.
	\end{align*}
	As long as we can show
	\begin{align}\label{eq:induction-2-key}
		&\sum_{m,n\ge 0} \frac{(-1)^n t^{m+n}q^{m^2+mn+\binom{n+1}{2}}(t;q)_{n}}{(q;q)_{N-m}(t^2q;q)_{N+n} (q;q)_{l_1-n}(tq;q)_{l_1+m} (q;q)_m(q;q)_{n}}\notag\\
		&\qquad\qquad\qquad\qquad = \frac{1}{(q;q)_N (tq;q)_N} \sum_{l_0\ge 0} \frac{t^{2l_0}q^{l_0^2}}{(q;q)_{l_1-l_0}(q;q)_{l_0}(t^2q;q)_{N+l_0}},
	\end{align}
	then \eqref{eq:Z-expression-new} becomes valid. Note that
	\begin{align*}
		\LHS\eqref{eq:induction-2-key} = \sum_{m\ge 0} \frac{t^{m}q^{m^2}}{(q;q)_{N-m} (q;q)_m(tq;q)_{l_1+m}} \sum_{n\ge 0} \frac{(-1)^n t^{n}q^{\binom{n}{2}+(m+1)n}(t;q)_{n}}{(q;q)_{l_1-n}(q;q)_{n}(t^2q;q)_{N+n}}.
	\end{align*}
	For the inner sum over $n$, we have
	\begin{align*}
		&\sum_{n\ge 0} \frac{(-1)^n t^{n}q^{\binom{n}{2}+(m+1)n}(t;q)_{n}}{(q;q)_{l_1-n}(q;q)_{n}(t^2q;q)_{N+n}}\\
		&\quad = \frac{1}{(q;q)_{l_1}(t^2q;q)_N} {}_{2}\phi_{1} \left(\begin{matrix}
			t,q^{-l_1}\\
			t^2q^{N+1}
		\end{matrix};q,tq^{l_1+m+1}\right)\\
		&\quad = \frac{1}{(q;q)_{l_1}(t^2q;q)_N} \frac{(t^2q^{N+l_1+1},tq^{m+1};q)_\infty}{(t^2q^{N+1},tq^{l_1+m+1};q)_\infty} {}_{2}\phi_{1} \left(\begin{matrix}
			q^{-(N-m)},q^{-l_1}\\
			tq^{m+1}
		\end{matrix};q,t^2q^{N+l_1+1}\right)\\
		&\quad = \frac{(t^2q^{N+l_1+1},tq^{m+1};q)_\infty(q;q)_{N-m}}{(t^2q,tq^{l_1+m+1};q)_\infty}\\
		&\quad\quad\times \sum_{l_0\ge 0} \frac{t^{2l_0}q^{l_0^2+ml_0}}{(q;q)_{l_1-l_0}(q;q)_{N-m-l_0}(q;q)_{l_0}(tq^{m+1};q)_{l_0}},
	\end{align*}
	where we have applied Heine's second transform \eqref{eq:Heine2}. Hence,
	\begin{align*}
		\LHS\eqref{eq:induction-2-key}&= \frac{1}{(t^2q;q)_{N+l_1}}\sum_{l_0\ge 0} \frac{t^{2l_0}q^{l_0^2}}{(q;q)_{l_1-l_0}(q;q)_{l_0}(tq;q)_{l_0}}\\
		&\quad\times \sum_{m\ge 0} \frac{t^{m}q^{m^2+l_0m}}{(q;q)_{(N-l_0)-m} (q;q)_m(tq^{l_0+1};q)_{m}}\\
		&= \frac{1}{(t^2q;q)_{N+l_1}}\sum_{l_0\ge 0} \frac{t^{2l_0}q^{l_0^2}}{(q;q)_{N-l_0}(q;q)_{l_1-l_0}(q;q)_{l_0}(tq;q)_{l_0}}\\
		&\quad\times \lim_{\tau\to 0} {}_{2}\phi_{1} \left(\begin{matrix}
			1/\tau,q^{-(N-l_0)}\\
			tq^{l_0+1}
		\end{matrix};q,tq^{N+1}\tau\right)\\
		\text{\tiny (by \eqref{eq:qCV-1})} &= \frac{1}{(t^2q;q)_{N+l_1}}\sum_{l_0\ge 0} \frac{t^{2l_0}q^{l_0^2}}{(q;q)_{N-l_0}(q;q)_{l_1-l_0}(q;q)_{l_0}(tq;q)_{l_0}}\frac{1}{(tq^{l_0+1};q)_{N-l_0}}\\
		&= \frac{1}{(tq;q)_N(t^2q;q)_{N+l_1}}\frac{1}{(q;q)_N(q;q)_{l_1}}\\
		&\quad\times \lim_{\tau\to 0} {}_{2}\phi_{1} \left(\begin{matrix}
			q^{-N},q^{-l_1}\\
			t^2q\tau
		\end{matrix};q,t^2q^{N+l_1+1}\right)\\
		\text{\tiny (by \eqref{eq:Heine2})}&= \frac{1}{(tq;q)_N(t^2q;q)_{N+l_1}}\frac{1}{(q;q)_N(q;q)_{l_1}}\\
		&\quad\times \lim_{\tau\to 0} \frac{(t^2q^{l_1+1}\tau,t^2q^{N+1};q)_\infty}{(t^2q\tau,t^2q^{N+l_1+1};q)_\infty} {}_{2}\phi_{1} \left(\begin{matrix}
			1/\tau,q^{-l_1}\\
			t^2q^{N+1}
		\end{matrix};q,t^2q^{l_1+1}\tau\right)\\
		&= \frac{1}{(q;q)_N (tq;q)_N} \sum_{l_0\ge 0} \frac{t^{2l_0}q^{l_0^2}}{(q;q)_{l_1-l_0}(q;q)_{l_0}(t^2q;q)_{N+l_0}},
	\end{align*}
	as requested.
\end{proof}

To relate the sum in Theorem \ref{th:Z-expression-new} to that in Theorem \ref{th:Z-expression}, we require the following general result.

\begin{theorem}\label{th:old=new}
	For any nonnegative integer $N$,
	\begin{align}\label{eq:old=new}
		&\sum_{n_1,\ldots,n_k\ge 0} \frac{t^{\sum_{i=1}^k 2n_i} q^{\sum_{i=1}^k n_i^2}}{(aq;q)_{N-n_k}(q;q)_{n_k}(tq;q)_{n_1}} \qbinom{n_k}{n_{k-1}}_q\qbinom{n_{k-1}}{n_{k-2}}_q\cdots \qbinom{n_2}{n_1}_q\notag\\
		&\qquad\qquad = \frac{(at^2q;q)_\infty}{(tq;q)_\infty(aq;q)_N} \sum_{n_1,\ldots,n_k\ge 0} (-1)^{n_k} t^{-n_k+\sum_{i=1}^k 2n_i} q^{-\binom{n_k}{2}+\sum_{i=1}^k n_i^2}\notag\\
		&\qquad\qquad\quad\times \frac{(t;q)_{n_k}}{(q;q)_{n_k}(at^2q;q)_{N+n_1}} \qbinom{n_k}{n_{k-1}}_q\qbinom{n_{k-1}}{n_{k-2}}_q\cdots \qbinom{n_2}{n_1}_q.
	\end{align}
\end{theorem}

Before providing its proof, we refresh our memory of the connection between \eqref{eq:Z1-1} and \eqref{eq:Z1-2}. What we have done is the identity
\begin{align*}
	\frac{1}{(tq;q)_N}\sum_{n\ge 0} \frac{(-1)^{n} t^{n} q^{\binom{n+1}{2}} (t;q)_{n}}{(q;q)_{N-n} (q;q)_{n}} = \sum_{n\ge 0} \frac{t^{2n} q^{n^2}}{(q;q)_{N-n} (q;q)_{n} (tq;q)_{n}}.
\end{align*}
Now we shall go slightly further.

\begin{lemma}\label{le:Z-1=Z-2-full}
	For any nonnegative integers $L$, $M$ and $N$,
	\begin{align}\label{eq:Z-1=Z-2-full}
		&\sum_{n\ge L} \frac{(-1)^{n} t^{2n} q^{\binom{n+1}{2}} b^{-n}(b;q)_{n}}{(q;q)_{M-n} (q;q)_{n-L} (at^2q;q)_{N+n}}\notag\\
		&\qquad\qquad=\frac{(-1)^L q^{-\binom{L}{2}}b^{-L}(b;q)_L (aq;q)_{N-L}(b^{-1}t^2q;q)_M}{(at^2q;q)_{N+M}} \notag\\
		&\qquad\qquad\quad\times \sum_{n\ge L} \frac{t^{2n}q^{n^2}}{(aq;q)_{N-n}(q;q)_{M-n}(q;q)_{n-L} (b^{-1}t^2q;q)_{n}}.
	\end{align}
\end{lemma}

In \eqref{eq:Z-1=Z-2-full}, we may put $b=1/\tau$ and take the limit at $\tau\to 0$:
\begin{align*}
	&\sum_{n\ge L} \frac{t^{2n} q^{n^2}}{(q;q)_{M-n} (q;q)_{n-L} (at^2q;q)_{N+n}}\\
	&\qquad\qquad = \frac{(aq;q)_{N-L}}{(at^2q;q)_{N+M}} \sum_{n\ge L} \frac{t^{2n}q^{n^2}}{(aq;q)_{N-n}(q;q)_{M-n}(q;q)_{n-L}}.
\end{align*}
We shall refer to this process as ``\emph{taking $b=\infty$}.'' Meanwhile, we may also take the limit at $M\to \infty$ in \eqref{eq:Z-1=Z-2-full}:
\begin{align*}
	&\sum_{n\ge L} \frac{(-1)^{n} t^{2n} q^{\binom{n+1}{2}} b^{-n}(b;q)_{n}}{(q;q)_{n-L} (at^2q;q)_{N+n}}\notag\\
	&\qquad\qquad=\frac{(-1)^L q^{-\binom{L}{2}}b^{-L}(b;q)_L (aq;q)_{N-L}(b^{-1}t^2q;q)_\infty}{(at^2q;q)_{\infty}} \notag\\
	&\qquad\qquad\quad\times \sum_{n\ge L} \frac{t^{2n}q^{n^2}}{(aq;q)_{N-n}(q;q)_{n-L} (b^{-1}t^2q;q)_{n}}.
\end{align*}
This process will be known as ``\emph{taking $M=\infty$}.''

\begin{proof}[Proof of Lemma \ref{le:Z-1=Z-2-full}]
	We have
	\begin{align*}
		&\!\!\!\!\!\!\!\!\!\!\LHS\eqref{eq:Z-1=Z-2-full}\\
		&= \frac{(-1)^L t^{2L} q^{\binom{L+1}{2}}b^{-L}(b;q)_L}{(at^2q;q)_{N+L}}\sum_{n\ge 0} \frac{(-1)^{n} t^{2n} q^{\binom{n+1}{2}+Ln} b^{-n} (bq^L;q)_{n}}{(q;q)_{(M-L)-n} (q;q)_{n} (at^2q^{N+L+1};q)_n}\\
		&= \frac{(-1)^L t^{2L} q^{\binom{L+1}{2}}b^{-L}(b;q)_L}{(at^2q;q)_{N+L} (q;q)_{M-L}} {}_{2}\phi_{1} \left(\begin{matrix}
			bq^L,q^{-(M-L)}\\
			at^2q^{N+L+1}
		\end{matrix};q,b^{-1}t^2q^{M+1}\right)\\
		\text{\tiny (by \eqref{eq:Heine2})} &= \frac{(-1)^L t^{2L} q^{\binom{L+1}{2}}b^{-L}(b;q)_L}{(at^2q;q)_{N+L}(q;q)_{M-L}} \frac{(at^2q^{N+M+1},b^{-1}t^2q^{L+1};q)_\infty}{(at^2q^{N+L+1},b^{-1}t^2q^{M+1};q)_\infty}\\
		&\quad\times {}_{2}\phi_{1} \left(\begin{matrix}
			a^{-1}q^{-(N-L)},q^{-(M-L)}\\
			b^{-1}t^2q^{L+1}
		\end{matrix};q,at^2q^{N+M+1}\right)\\
		&= \frac{(-1)^L t^{2L} q^{\binom{L+1}{2}}b^{-L}(b;q)_L (aq;q)_{N-L}(b^{-1}t^2q;q)_M}{(at^2q;q)_{N+M}}\\
		&\quad\times \sum_{n\ge 0} \frac{t^{2n}q^{n^2+2Ln}}{(aq;q)_{N-L-n}(q;q)_{M-L-n}(q;q)_n (b^{-1}t^2q;q)_{L+n}}\\
		&= \frac{(-1)^L q^{-\binom{L}{2}}b^{-L}(b;q)_L (aq;q)_{N-L}(b^{-1}t^2q;q)_M}{(at^2q;q)_{N+M}}\\
		&\quad\times \sum_{n\ge L} \frac{t^{2n}q^{n^2}}{(aq;q)_{N-n}(q;q)_{M-n}(q;q)_{n-L} (b^{-1}t^2q;q)_{n}},
	\end{align*}
	as desired.
\end{proof}

We are then in a position to prove Theorem \ref{th:old=new}.

\begin{proof}[Proof of Theorem \ref{th:old=new}]
	It is clear that the $k=1$ case of \eqref{eq:old=new} is
	\begin{align*}
		&\sum_{n_1 \ge 0} \frac{t^{2n_1} q^{n_1^2}}{(aq;q)_{N-n_1}(q;q)_{n_1}(tq;q)_{n_1}}\\
		&\qquad = \frac{(at^2q;q)_\infty}{(tq;q)_\infty(aq;q)_N} \sum_{n_1\ge 0} \frac{(-1)^{n_1} t^{n_1} q^{\binom{n_1+1}{2}} (t;q)_{n_1}}{(q;q)_{n_1} (at^2q;q)_{N+n_1}};
	\end{align*}
	this is exactly \eqref{eq:Z-1=Z-2-full} with $(b,L,M) = (t,0,\infty)$. Now we assume that $k\ge 2$ and begin with the right-hand side of \eqref{eq:old=new} by singling out the sum over $n_1$:
	\begin{align*}
		&\RHS\eqref{eq:old=new}\\
		&= \frac{(at^2q;q)_\infty}{(tq;q)_\infty(aq;q)_N} \sum_{n_2,\ldots,n_k\ge 0} \frac{(-1)^{n_k} t^{-n_k+\sum_{i=2}^k 2n_i} q^{-\binom{n_k}{2}+\sum_{i=2}^k n_i^2} (t;q)_{n_k}}{(q;q)_{n_k-n_{k-1}}\cdots (q;q)_{n_3-n_2}}\\
		&\quad\times \sum_{n_1\ge 0} \frac{t^{2n_1}q^{n_1^2}}{(q;q)_{n_2-n_1}(q;q)_{n_1}(at^2q;q)_{N+n_1}}.
	\end{align*}
	We then apply \eqref{eq:Z-1=Z-2-full} with $(b,L,M) = (\infty,0,n_2)$ to this sum over $n_1$ to derive
	\begin{align*}
		&\RHS\eqref{eq:old=new}\\
		&= \frac{(at^2q;q)_\infty}{(tq;q)_\infty} \sum_{n_1\ge 0} \frac{t^{2n_1}q^{n_1^2}}{(q;q)_{n_1}}\sum_{n_3,\ldots,n_k\ge 0} \frac{(-1)^{n_k} t^{-n_k+\sum_{i=3}^k 2n_i} q^{-\binom{n_k}{2}+\sum_{i=3}^k n_i^2} (t;q)_{n_k}}{(q;q)_{n_k-n_{k-1}}\cdots (q;q)_{n_4-n_3}}\\
		&\quad\times \frac{1}{(aq;q)_{N-n_1}}\sum_{n_2\ge n_1} \frac{t^{2n_2}q^{n_2^2}}{(q;q)_{n_3-n_2}(q;q)_{n_2-n_1}(at^2q;q)_{N+n_2}}.
	\end{align*}
	We continue to use \eqref{eq:Z-1=Z-2-full} with $(b,L,M) = (\infty,n_1,n_3)$ to this sum over $n_2$. In general, we sequentially apply \eqref{eq:Z-1=Z-2-full} with $(b,L,M) = (\infty,n_{i-1},n_{i+1})$ to the sum over $n_i$ for $i=2,\ldots,k-1$. Thus,
	\begin{align*}
		\RHS\eqref{eq:old=new} &= \frac{(at^2q;q)_\infty}{(tq;q)_\infty} \sum_{n_1,\ldots,n_{k-1}\ge 0} \frac{t^{\sum_{i=1}^{k-1} 2n_i} q^{\sum_{i=1}^{k-1} n_i^2}}{(q;q)_{n_{k-1}-n_{k-2}}\cdots (q;q)_{n_2-n_1} (q;q)_{n_1}}\\
		&\quad\times \frac{1}{(aq;q)_{N-n_{k-1}}}\sum_{n_k\ge n_{k-1}} \frac{(-1)^{n_k} t^{n_k}q^{\binom{n_k+1}{2}} (t;q)_{n_k}}{(q;q)_{n_k-n_{k-1}}(at^2q;q)_{N+n_k}}.
	\end{align*}
	For the sum over $n_k$, we apply \eqref{eq:Z-1=Z-2-full} with $(b,L,M) = (t,n_{k-1},\infty)$ and get
	\begin{align}\label{eq:old-new-middle}
		\RHS\eqref{eq:old=new} &= \sum_{n_k\ge 0} \frac{t^{n_k}q^{n_k^2}}{(aq;q)_{N-n_k}(tq;q)_{n_k}}\notag\\
		&\quad\times \sum_{n_1,\ldots,n_{k-1}\ge 0} \frac{(-1)^{n_{k-1}} t^{-n_{k-1}+\sum_{i=1}^{k-1} 2n_i} q^{-\binom{n_{k-1}}{2}+\sum_{i=1}^{k-1} n_i^2} (t;q)_{n_{k-1}}}{(q;q)_{n_k-n_{k-1}}\cdots (q;q)_{n_2-n_1}(q;q)_{n_1}}.
	\end{align}
	Now we single out the sum over $n_{k-1}$:
	\begin{align*}
		\RHS\eqref{eq:old=new} &= \sum_{n_k\ge 0} \frac{t^{n_k}q^{n_k^2}}{(aq;q)_{N-n_k}}\sum_{n_1,\ldots,n_{k-2}\ge 0} \frac{t^{\sum_{i=1}^{k-2} 2n_i} q^{\sum_{i=1}^{k-2} n_i^2}}{(q;q)_{n_{k-2}-n_{k-3}}\cdots (q;q)_{n_2-n_1}(q;q)_{n_1}}\\
		&\quad\times \frac{1}{(tq;q)_{n_k}} \sum_{n_{k-1}\ge n_{k-2}} \frac{(-1)^{n_{k-1}} t^{n_{k-1}} q^{\binom{n_{k-1}+1}{2}} (t;q)_{n_{k-1}}}{(q;q)_{n_k-n_{k-1}}(q;q)_{n_{k-1}-n_{k-2}}}.
	\end{align*}
	We then utilize \eqref{eq:Z-1=Z-2-full} with $(a,b,L,M) = (0,t,n_{k-2},n_{k})$ to this sum over $n_{k-1}$. In general, we take turns applying \eqref{eq:Z-1=Z-2-full} with $(a,b,L,M) = (0,t,n_{i-1},n_{i+1})$ to the sum over $n_i$ for $i=k-1,\ldots,2$. Hence,
	\begin{align*}
		\RHS\eqref{eq:old=new} &= \sum_{n_2,\ldots,n_{k}\ge 0} \frac{t^{\sum_{i=2}^{k} 2n_i} q^{\sum_{i=2}^{k} n_i^2}}{(aq;q)_{N-n_k}(q;q)_{n_{k}-n_{k-1}}\cdots (q;q)_{n_3-n_2}}\\
		&\quad\times \frac{1}{(tq;q)_{n_2}} \sum_{n_{1}\ge 0} \frac{(-1)^{n_{1}} t^{n_{1}} q^{\binom{n_{1}+1}{2}} (t;q)_{n_{1}}}{(q;q)_{n_2-n_{1}}(q;q)_{n_{1}}}.
	\end{align*}
	Finally, applying \eqref{eq:Z-1=Z-2-full} with $(a,b,L,M) = (0,t,0,n_{2})$ to the sum over $n_1$ yields the left-hand side of \eqref{eq:old=new}.
\end{proof}

Now Theorem \ref{th:Z-expression} becomes clear.

\begin{proof}[Proof of Theorem \ref{th:Z-expression}]
	The $k=1$ case has been shown in \eqref{eq:Z1-2}. For $k\ge 2$, we recall \eqref{eq:Z-expression-new} and use \eqref{eq:old=new} with $a=1$.
\end{proof}

It is also notable that from \eqref{eq:V-expression}, we may apply \eqref{eq:old=new} with $a=t^{-1}$ to derive the following alternative expression for $\cV_k(N;t,q)$.

\begin{theorem}
	For any nonnegative integer $N$,
	\begin{align}\label{eq:V-expression-new}
		\cV_k(N;t,q) &= \frac{(tq;q)_\infty (t^{-1}q;q)_N}{(q;q)_\infty (q;q)_N} \sum_{n_1,\ldots,n_k\ge 0} t^{\sum_{i=1}^k 2n_i} q^{\sum_{i=1}^k n_i^2}\notag\\
		&\quad\times \frac{1}{(t^{-1}q;q)_{N-n_k}(q;q)_{n_k}(tq;q)_{n_1}} \qbinom{n_k}{n_{k-1}}_q\qbinom{n_{k-1}}{n_{k-2}}_q\cdots \qbinom{n_2}{n_1}_q.
	\end{align}
\end{theorem}

\section{Theorem \ref{th:S-m-infty} revisited}

As the first application of \eqref{eq:Z-expression}, we revisit Theorem \ref{th:S-m-infty}, or more precisely, its finitization Theorem \ref{th:S-1-finite} concerning $\tcZ_k(N;1,q)$, and give an alternative proof.

\begin{proof}[Second proof of Theorem \ref{th:S-1-finite}]
	It follows from \eqref{eq:Z-expression} that
	\begin{align*}
		\tcZ_k(N;1,q) = \frac{1}{(q;q)_N} \sum_{n_1,\ldots,n_k\ge 0} \frac{q^{\sum_{i=1}^k n_i^2}}{(q;q)_{N-n_k}(q;q)_{n_k-n_{k-1}}\cdots (q;q)_{n_2-n_1}(q;q)_{n_1}^2}.
	\end{align*}
	Then we only need to apply \eqref{eq:multi-sum-N} with $d_1=\cdots = d_k = 0$ and $a=1$ to arrive at \eqref{eq:S-1-finite}.
\end{proof}

\section{Finitization of Theorems \ref{th:S-m-(-1)-infty} and \ref{th:strange-A2-sum}}

For the second application of \eqref{eq:Z-expression}, we complete the proof of Theorems \ref{th:S-m-(-1)-infty} and \ref{th:strange-A2-sum}. To begin with, we need the following single-sum expression for the finite multisum $\tcZ_k(N;-1,q)$.

\begin{theorem}
	For any nonnegative integer $N$,
	\begin{align}\label{eq:Z-(-1)-expression}
		\tcZ_k(N;-1,q) = \frac{1}{(q;q)_{2N}(-q;q)_N}\sum_{n=-N}^N (-1)^n q^{(k+1)n^2} \qbinom{2N}{N-n}_q.
	\end{align}
\end{theorem}

\begin{proof}
	In light of \eqref{eq:Z-expression},
	\begin{align*}
		&\tcZ_k(N;-1,q)\\
		&\qquad = \frac{1}{(-q;q)_N} \sum_{n_1,\ldots,n_k\ge 0} \frac{q^{\sum_{i=1}^k n_i^2}}{(q;q)_{N-n_k}(q;q)_{n_k-n_{k-1}}\cdots (q;q)_{n_2-n_1}(q^2;q^2)_{n_1}}.
	\end{align*}
	Meanwhile, it is a standard result on $\mathrm{A}_1$ Rogers--Ramanujan type identities \cite[p.~3]{War1999} that
	\begin{align*}
		&\sum_{n_1,\ldots,n_k\ge 0} \frac{q^{\sum_{i=1}^k n_i^2}}{(q;q)_{N-n_k}(q;q)_{n_k-n_{k-1}}\cdots (q;q)_{n_2-n_1}(q^2;q^2)_{n_1}} \\
		&\qquad = \frac{1}{(q;q)_{2N}}\sum_{n=-N}^N (-1)^n q^{(k+1)n^2} \qbinom{2N}{N-n}_q,
	\end{align*}
	which leads us to the claimed equality.
\end{proof}

The limiting case at $N\to \infty$ fills in the last piece of the puzzle.

\begin{proof}[Proof of Theorems \ref{th:S-m-(-1)-infty} and \ref{th:strange-A2-sum}]
	Recalling \eqref{eq:Z-new}, we have
	\begin{align*}
		\RHS\eqref{eq:strange-A2-sum} &= (-q;q)_\infty^2 \lim_{N\to \infty} \tcZ_k(N;-1,q)\\
		\text{\tiny (by \eqref{eq:Z-(-1)-expression})} &= (-q;q)_\infty^2 \cdot \frac{1}{(q;q)_\infty^2 (-q;q)_\infty} \sum_{n=-\infty}^\infty (-1)^n q^{(k+1)n^2}\\
		\text{\tiny (by \eqref{eq:JTP})} &= \frac{(q^2;q^2)_\infty (q^{k+1};q^{k+1})_\infty^2}{(q;q)_\infty^3 (q^{2k+2};q^{2k+2})_\infty};
	\end{align*}
	this is the left-hand side of \eqref{eq:strange-A2-sum}. In the meantime, we know from \eqref{eq:Z-finite-limit} that
	\begin{align*}
		\cZ_k(-1,q) = (q;q)_\infty \lim_{N\to \infty} \tcZ_k(N;-1,q),
	\end{align*}
	and hence arrive at \eqref{eq:S-m-(-1)-infty}.
\end{proof}

\section{Huang and Jiang's reflection formula}

Our last episode revolves around Huang and Jiang's reflection formula in Theorem \ref{th:HJ-transform}.

\begin{proof}[Proof of Theorem \ref{th:HJ-transform}]
	In view of \eqref{eq:Z-expression},
	\begin{align*}
		&\tcZ_k(N;t^{-1}q^{-N},q)\\
		&\qquad = \frac{1}{(t^{-1}q^{1-N};q)_N} \sum_{n_1,\ldots,n_k\ge 0} t^{-\sum_{i=1}^k 2n_i} q^{\sum_{i=1}^k (n_i^2-2Nn_i)}\notag\\
		&\qquad\quad\times \frac{1}{(q;q)_{N-n_k}(q;q)_{n_k-n_{k-1}}\cdots (q;q)_{n_2-n_1}(q;q)_{n_1}(t^{-1}q^{1-N};q)_{n_1}}.
	\end{align*}
	Now we make the changes of variables for $1\le i\le k$:
	\begin{align*}
		n_i \mapsto N-n_{k+1-i}.
	\end{align*}
	Then,
	\begin{align*}
		&\tcZ_k(N;t^{-1}q^{-N},q)\\
		&\qquad = \frac{t^{-2kN}q^{-kN^2}}{(t^{-1}q^{1-N};q)_N} \sum_{n_1,\ldots,n_k\ge 0} t^{\sum_{i=1}^k 2n_i} q^{\sum_{i=1}^k n_i^2}\notag\\
		&\qquad\quad\times \frac{1}{(t^{-1}q^{1-N};q)_{N-n_k}(q;q)_{N-n_k}(q;q)_{n_k-n_{k-1}}\cdots (q;q)_{n_2-n_1}(q;q)_{n_1}},
	\end{align*}
	so that
	\begin{align*}
		&\frac{(1-t)^2 q^N (t^{2N} q^{N^2})^{k-1}}{(1-t q^N)^2} \tcZ_k(N;t^{-1}q^{-N},q)\\
		&\qquad = \frac{1}{(tq;q)_N^2} \sum_{n_1,\ldots,n_k\ge 0} \frac{(-1)^{n_k} t^{-n_k+\sum_{i=1}^k 2n_i} q^{-\binom{n_k}{2}+\sum_{i=1}^k n_i^2} (t;q)_{n_k}}{(q;q)_{N-n_k}(q;q)_{n_k-n_{k-1}}\cdots (q;q)_{n_2-n_1}(q;q)_{n_1}}.
	\end{align*}
	Hence, our task becomes to show that
	\begin{align}\label{eq:HJ-new}
		&\frac{1}{(tq;q)_N} \sum_{n_1,\ldots,n_k\ge 0} \frac{(-1)^{n_k} t^{-n_k+\sum_{i=1}^k 2n_i} q^{-\binom{n_k}{2}+\sum_{i=1}^k n_i^2} (t;q)_{n_k}}{(q;q)_{N-n_k}(q;q)_{n_k-n_{k-1}}\cdots (q;q)_{n_2-n_1}(q;q)_{n_1}}\notag\\
		&\qquad = \sum_{n_1,\ldots,n_k\ge 0} \frac{t^{\sum_{i=1}^k 2n_i} q^{\sum_{i=1}^k n_i^2}}{(q;q)_{N-n_k}(q;q)_{n_k-n_{k-1}}\cdots (q;q)_{n_2-n_1}(q;q)_{n_1}(tq;q)_{n_1}}.
	\end{align}
	For the left-hand side of \eqref{eq:HJ-new}, we single out the sum over $n_k$:
	\begin{align*}
		\LHS\eqref{eq:HJ-new}&= \sum_{n_1,\ldots,n_{k-1}\ge 0} \frac{t^{\sum_{i=1}^{k-1} 2n_i} q^{\sum_{i=1}^{k-1} n_i^2}}{(q;q)_{n_{k-1}-n_{k-2}}\cdots (q;q)_{n_2-n_1}(q;q)_{n_1}}\\
		&\quad\times \frac{1}{(tq;q)_N} \sum_{n_k\ge n_{k-1}} \frac{(-1)^{n_k} t^{n_k} q^{\binom{n_k+1}{2}} (t;q)_{n_k}}{(q;q)_{N-n_k} (q;q)_{n_{k}-n_{k-1}}}.
	\end{align*}
	Using \eqref{eq:Z-1=Z-2-full} with $(a,b,L,M) = (0,t,n_{k-1},N)$ to this sum over $n_{k}$ implies that
	\begin{align*}
		\LHS\eqref{eq:HJ-new}&= \sum_{n_k\ge 0} \frac{t^{n_k}q^{n_k^2}}{(q;q)_{N-n_k}(tq;q)_{n_k}}\notag\\
		&\quad\times \sum_{n_1,\ldots,n_{k-1}\ge 0} \frac{(-1)^{n_{k-1}} t^{-n_{k-1}+\sum_{i=1}^{k-1} 2n_i} q^{-\binom{n_{k-1}}{2}+\sum_{i=1}^{k-1} n_i^2} (t;q)_{n_{k-1}}}{(q;q)_{n_k-n_{k-1}}\cdots (q;q)_{n_2-n_1}(q;q)_{n_1}},
	\end{align*}
	which is exactly the right-hand side of \eqref{eq:old-new-middle} with $a=1$. Due to the equality between \eqref{eq:old-new-middle} and both sides of \eqref{eq:old=new}, the right-hand side of \eqref{eq:HJ-new} becomes the final output.
\end{proof}

\subsection*{Acknowledgements}

I would like to thank Yifeng Huang for introducing the conjectures in \cite{HJ2023} to me. This work was supported by the Austrian Science Fund (No.~10.55776/F1002).

\bibliographystyle{amsplain}

\begin{thebibliography}{99}
	
	\bibitem{AB2004}
	K. Alladi and A. Berkovich, New polynomial analogues of Jacobi's triple product and Lebesgue's identities, \textit{Adv. in Appl. Math.} \textbf{32} (2004), no. 4, 801--824.
	
	\bibitem{And1974}
	G. E. Andrews, An analytic generalization of the Rogers--Ramanujan identities for odd moduli, \textit{Proc. Nat. Acad. Sci. U.S.A.} \textbf{71} (1974), 4082--4085.
	
	\bibitem{And1986}
	G. E. Andrews, \textit{$q$-Series: their development and application in analysis, number theory, combinatorics, physics, and computer algebra}, American Mathematical Society, Providence, RI, 1986.
	
	\bibitem{And1998}
	G. E. Andrews, \textit{The theory of partitions}, Cambridge University Press, Cambridge, 1998.
	
	\bibitem{ASW1999}
	G. E. Andrews, A. Schilling, and S. O. Warnaar, An $\mathrm{A}_2$ Bailey lemma and Rogers--Ramanujan-type identities, \textit{J. Amer. Math. Soc.} \textbf{12} (1999), no. 3, 677--702.
	
	\bibitem{Bif1989}
	E. Bifet, Sur les points fixes du sch\'ema $\operatorname{Quot}_{\mathscr{O}_X^r/X/k}$ sous l'action du tore $\mathbf{G}_{m,k}^r$, \textit{C. R. Acad. Sci. Paris S\'er. I Math.} \textbf{309} (1989), no. 9, 609--612.
	
	\bibitem{CL1984}
	H. Cohen and H. W. Lenstra Jr., Heuristics on class groups of number fields, in: \textit{Number theory, Noordwijkerhout 1983}, 33--62, Springer, Berlin, 1984.
	
	\bibitem{GR2004}
	G. Gasper and M. Rahman, \textit{Basic hypergeometric series. Second edition}, Cambridge University Press, Cambridge, 2004.
	
	\bibitem{Gor1961}
	B. Gordon, A combinatorial generalization of the Rogers--Ramanujan identities, \textit{Amer. J. Math.} \textbf{83} (1961), 393--399.
	
	\bibitem{Hua2023}
	Y. Huang, Mutually annihilating matrices, and a Cohen-Lenstra series for the nodal singularity, \textit{J. Algebra} \textbf{619} (2023), 26--50.
	
	\bibitem{HJ2023}
	Y. Huang and R. Jiang, Generating series for torsion-free bundles over singular curves: rationality, duality and modularity, preprint, 2023. Available at arXiv:2312.12528.
	
	\bibitem{KT2023}
	O. Kivinen and M.-T. Q. Trinh, The Hilb-vs-Quot conjecture, preprint, 2023. Available at arXiv:2310.19633.
	
	\bibitem{Mac1972}
	I. G. Macdonald, Affine root systems and Dedekind's $\eta$-function, \textit{Invent. Math.} \textbf{15} (1972), 91--143.
	
	\bibitem{Pau1985}
	P. Paule, On identities of the Rogers--Ramanujan type, \textit{J. Math. Anal. Appl.} \textbf{107} (1985), no. 1, 255--284.
	
	\bibitem{Ram1914}
	S. Ramanujan, Problem 584, \textit{J. Indian Math. Soc.} \textbf{6} (1914), 199--200.
	
	\bibitem{RR1919}
	S. Ramanujan and L. J. Rogers, Proof of certain identities in combinatory analysis, \textit{Cambr. Phil. Soc. Proc.} \textbf{19} (1919), 211--216.
	
	\bibitem{Rog1894}
	L. J. Rogers, Second memoir on the expansion of certain infinite products, \textit{Proc. Lond. Math. Soc.} \textbf{25} (1893/94), 318--343.
	
	\bibitem{Row2010}
	M. Rowell, A new exploration of the Lebesgue identity, \textit{Int. J. Number Theory} \textbf{6} (2010), no. 4, 785--798.
	
	\bibitem{SS2002}
	J. P. O. Santos and A. V. Sills, $q$-Pell sequences and two identities of V. A. Lebesgue, \textit{Discrete Math.} \textbf{257} (2002), no. 1, 125--142.
	
	\bibitem{Sch1917}
	I. Schur, Ein Beitrag zur additiven Zahlentheorie und zur Theorie der Kettenbr\"uche, \textit{Sitzungsber. Preuss. Akad. Wiss. Phys.-Math. Klasse} (1917), 302--321.
	
	\bibitem{War1999}
	S. O. Warnaar, Supernomial coefficients, Bailey's lemma and Rogers-Ramanujan-type identities. A survey of results and open problems, \textit{S\'em. Lothar. Combin.} \textbf{42} (1999), Paper No. B42n, 22 pp.
	
	\bibitem{War2023}
	S. O. Warnaar, The $\mathrm{A}_2$ Andrews--Gordon identities and cylindric partitions, \textit{Trans. Amer. Math. Soc. Ser. B} \textbf{10} (2023), 715--765.
	
\end{thebibliography}

\end{document}